\documentclass[12pt]{article}

\usepackage{amsmath}    % need for subequations
\usepackage{amssymb}
\usepackage{graphicx}   % need for figures
\usepackage{verbatim}   % useful for program listings
\usepackage{color}      % use if colour is used in text

\usepackage{hyperref}   % use for hypertext links, including those to external documents and URLs

\usepackage{enumerate}
\usepackage{mathrsfs}
\usepackage{empheq}
\usepackage{bbold}

\usepackage[final]{showlabels}

\usepackage{amsthm}
\newtheorem{definition}{Definition}[section]
\newtheorem{lemma}{Lemma}[section]
\newtheorem{corollary}{Corollary}[section]
\newtheorem{theorem}{Theorem}[section]

\newcommand{\del}{\partial}
\renewcommand{\theta}{\vartheta}
\renewcommand{\phi}{\varphi}

\newcommand{\veccc}[3]{\left ( \begin{array}{c}#1\\#2\\#3\\ \end{array}\right )}

\newcommand{\dd}{\mathrm{d}}

\newcommand{\id}{\mathbb{1}}
\newcommand{\grad}{\mathrm{grad\,}}
\renewcommand{\div}{\mathrm{div\,}}
\renewcommand{\vec}{\mathbf}

\newcommand{\sign}{\mathrm{\,sign\,}}

\newcommand{\ii}{\mathbb{i}}

\newcommand{\cfl}{\textsc{cfl }}
\newcommand{\regdist}[1]{\underbracket[0.5pt][2pt]{\,#1\,}}
\newcommand{\test}{\psi}

\begin{document}

\begin{center} \Large
Exact solution and the multidimensional Godunov scheme for the acoustic equations

%\shorttitle{Exact solution and a Godunov scheme for acoustics}
\vspace{1cm}

\date{}
\normalsize

%\begin{tabular}{cc}
Wasilij Barsukow\footnote{Institute for Mathematics, Zurich University, 8057 Zurich, Switzerland}, 
Christian Klingenberg\footnote{Institute for Mathematics, Wuerzburg University, Emil-Fischer-Strasse 40, 97074 Wuerzburg, Germany}
\end{center}

% Short list of authors for running heads:
%\shortauthorlist{W. Barsukow and C. Klingenberg}

%\maketitle

\begin{abstract}
The acoustic equations derived as a linearization of the Euler equations are a valuable system for studies of multi-dimensional solutions. Additionally they possess a low Mach number limit analogous to that of the Euler equations. Aiming at understanding the behaviour of the multi-dimensional Godunov scheme in this limit, first the exact solution of the corresponding Cauchy problem in three spatial dimensions is derived. The appearance of logarithmic singularities in the exact solution of the 4-quadrant Riemann Problem in two dimensions is discussed. The solution formulae are then used to obtain the multidimensional Godunov finite volume scheme in two dimensions. It is shown to be superior to the dimensionally split upwind/Roe scheme concerning its domain of stability and ability to resolve multi-dimensional Riemann problems. It is shown experimentally and theoretically that despite taking into account multi-dimensional information it is, however, not able to resolve the low Mach number limit.
\end{abstract}

\section{Introduction}

Hyperbolic systems of PDEs in multiple spatial dimensions exhibit a richer phenomenology than their one-dimensional counterparts. In the context of ideal hydrodynamics the most prominent such feature is vorticity. Vortical structures appear virtually everywhere in multi-dimensional flows, for example also in regions of hydrodynamical instability. Nontrivial incompressible flows also only exist in multiple spatial dimensions.

Numerical methods should reproduce such features. The methods that are dealt with in the paper are finite volume methods. They interpret the discrete degrees of freedom as averages over the computational cells. The temporal evolution of the averages is given by fluxes over the cell boundaries. One possibility to deal with the multi-dimensionality is to compute these fluxes using only one-dimensional information perpendicular to the cell boundary. Thus, for the complete cell update one-dimensional information is collected from different directions. Such methods are called \emph{dimensionally split}, and are widely used due to their simplicity.

Several shortcomings of such methods have been noticed e.g. in \cite{morton01,guillard04,dellacherierieper10}. They concern the treatment of vorticity and the incompressible (low Mach number) limit, i.e. a bad resolution of multi-dimensional features. Therefore it has been suggested to incorporate truly multi-dimensional information into the finite volume methods in a variety of ways: \cite{colella90,morton01,balsara12,leveque97,fey98a,fey98b,roe17} and many others. Attempts to construct Godunov schemes using exact solutions of multi-dimensional Riemann Problems (see e.g. \cite{zheng12}) suffer from the high complexity of the occurring solutions.

A path that circumvents solving the Euler equations directly has been suggested e.g. in \cite{godlewski13,chalons13,roe17}. The advective operator contained in the Euler equations is taken into account differently than the rest of the equations (which often is called \emph{acoustic operator}). Whereas the solution to advection in multiple dimensions is not very different from its one-dimensional counterpart, acoustics exhibits a number of new features. This has led \cite{abgrall93,li02,morton01,lukacova04a,amadori2015,dellacherierieper10,barsukow17lilleproceeding,barsukow17a} and others to studies of the linearized acoustic operator. 

This paper aims at deriving the multi-dimensional Godunov scheme for linear acoustics {using the exact multi-dimensional solution, because -- even if it is more complicated -- one might expect it to fulfill more of the fundamental properties of the PDE.} The derivation of any numerical scheme makes use of a number of approximations, and we aim at excluding as many as possible. The studies of this paper shall in particular help understanding the influence of the approximations on the ability of the scheme to resolve the low Mach number limit. 

{It has been shown in \cite{guillard04} that an exact one-dimensional Riemann solver introduces terms, which spoil the numerical solution in the low Mach number limit -- despite the fact that the solution of the Riemann problem is exact. However, in one spatial dimension the incompressible limit is trivial, and there are also no low Mach number artefacts (\cite{dellacherierieper10}). The low Mach number artefacts found in \cite{guillard04} appear in multiple spatial dimensions when the Riemann solver is used in a dimensionally split way. In view of the multi-dimensional nature of the low Mach number limit, the natural question therefore is whether the dimensional splitting itself is responsible for the artefacts -- because otherwise the Riemann solver is exact and there is no other approximation apart from the choice of the reconstruction. This is supported by the existence of multi-dimensional methods which are low Mach number compliant without any fix (\cite{barsukow17a}). Unfortunately, these methods lack a derivation from first principles, which is commonly considered an advantage of Godunov schemes. This paper therefore studies a multi-dimensional Godunov scheme thus excluding both dimensional splitting and approximate evolution as reasons for the failure to resolve the low Mach number limit. The remaining approximations are of a fundamental nature and cannot be removed without questioning the Godunov procedure itself.}

In order to follow the \emph{reconstruction}--\emph{evolution}--\emph{average} strategy the exact solution of the corresponding Cauchy problem shall be used. In other words, the aim is to have a numerical scheme where the evolution operator is exact.

The exact solution may be expressed in various representations, which differ by their applicability to the derivation of a Godunov scheme. Using bicharacteristics, for example, in \cite{ostkamp97,lukacova04a} analytical relations of the shape
\begin{align}
 q(t + \Delta t) = L_1[q(t)] \,\,+\!\!\! \int\limits_t^{t + \Delta t} \!\!\! \dd \tau \, L_2[q(\tau)] \label{eq:bicharimplicit}
\end{align}
are derived, which connect the solution $q(t + \Delta t)$ at time $t + \Delta t$ with the data $q(t)$ at time $t$ via a so called \emph{mantle integral}. This latter involves the solution at all intermediate times. $L_1$ and $L_2$ are certain linear operators, see e.g. Equations (2.14)--(2.16) in \cite{lukacova04a} for details. E.g. in \cite{lukacova04a}, numerical schemes are derived by carefully \emph{approximating} the integral in \eqref{eq:bicharimplicit} (e.g. in \cite{ostkamp97}, or \cite{lukacova04}, Equations (4.12)-(4-15), or \cite{lukacova00}, Equations (4.13)--(4.15)). Moreover, in order to extend the methodology to the Euler equations, a local linearization is employed (e.g. in \cite{lukacova04a}, p. 18).

The Godunov scheme studied here shall use an \emph{exact} evolution operator and the only approximation is in the reconstruction step. Therefore formulae derived and used in \cite{lukacova04,lukacova04a} cannot be employed here. An exact solution operator is needed which relates the solution at time $t + \Delta t$ solely to the data at time $t$.

Such operators have appeared in \cite{eymann13,roe17a,franck17} under the assumption of smooth initial data. Having the Godunov scheme in mind, however, it is necessary to study the Cauchy problem with \emph{discontinuous} initial data. In order to achieve this, it turns out to be necessary to consider distributional solutions. This paper thus for the first time gives a detailed derivation of the distributional solution to the Cauchy problem of linear acoustics, without the restriction to smooth initial data. In this paper a formula is obtained for the exact evolution which expresses the solution at time $t$ in function of the initial data at $t=0$ directly. 

Our exact solution formulae are interesting from a theoretical viewpoint as well. Before using them for the derivation of a numerical scheme, some of the analytical properties of solutions to linear acoustics are studied. The new formulae are applied to a two-dimensional Riemann Problem for linear acoustics. In \cite{lukacova03} the solution outside the sonic circle is presented. Based on \cite{abgrall93,gilquin93b,gilquin96}, in \cite{li02} and \cite{amadori2015} (Chapter 6.2) the solution inside the sonic circle is derived for linear acoustics using a self-similarity ansatz. It has been found that the velocity of the solution, in general, is unbounded at the origin, and that measure-valued (Dirac delta) vorticity appears. This demonstrates once more the need to interpret the solution to a multi-dimensional Riemann problem as a distributional one. In this paper we justify the observed singularity by using the framework of distributional solutions. It is shown that the singularity of the velocity is a logarithmic one, and its precise shape is obtained.

The exact solution is finally used to derive a Godunov method according to the \emph{reconstruction}--\emph{evolution}--\emph{average} strategy. This demonstrates that the exact solution formulae, despite their complexity, can be efficiently used to derive numerical schemes, of which the Godunov scheme is only one example. {A similar approach has been followed in \cite{brio01}. Here, the Godunov scheme is derived using the framework of distributional solutions, as the initial data of a Riemann Problem are discontinuous. Focusing on the acoustic equations allows to study the effects of an exact \emph{evolution} step.

In order to reduce the work necessary for the derivation of the Godunov scheme, it is first shown that for linear systems, evolution and average can be interchanged. The new strategy \emph{reconstruction}--\emph{average}--\emph{evolution} leads to the same scheme, but shortens the derivation considerably. A piecewise constant reconstruction is thus found to be equivalent to a staggered bilinear reconstruction endowed with a different interpretation of the discrete degree of freedom. This should not, however, be confused with bilinear reconstructions aimed at deriving schemes of second order.}

The paper is organized as follows. After deriving the equations of linear acoustics in Section \ref{sec:equations}, an exact solution in three spatial dimensions is presented in Section \ref{sec:solution}. The solution operator needs to allow for discontinuous initial data. It is shown that the natural class in multiple spatial dimensions are distributional solutions. A brief review of distributions is given in the Appendix in Section \ref{ssec:distributions}, followed by a detailed derivation in Section \ref{sec:distributionderivation} of the Appendix. The properties of the solution operator are discussed in Section \ref{ssec:properties}, as it has a number of striking differences to its one-dimensional counterpart. Section \ref{sec:riemann} exemplifies the formulae on a two-dimensional Riemann Problem and in Section \ref{sec:godunov} a Godunov method is derived according to the \emph{reconstruction}--\emph{evolution}--\emph{average} strategy. Numerical examples are shown in Section \ref{ssec:numericalresults}. The ability of the method to resolve the low Mach number limit is studied both theoretically and experimentally there as well.

\section{Acoustic equations} \label{sec:equations}

\subsection{Linearization of the Euler equations} \label{ssec:equationslinearization}

The acoustic equations are obtained as linearizations of the Euler equations. These latter govern the motion of an ideal compressible fluid. In $d$ spatial dimensions, the state of the fluid is given by specifying a density $\rho: \mathbb R^+_0 \times \mathbb R^d \to \mathbb R^+$ and a velocity $\vec v: \mathbb R^+_0 \times \mathbb R^d \to \mathbb R^d$. Additionally, the pressure $p$ of the fluid is needed to close the system. Its role depends on the precise model of the fluid motion.

Consider the isentropic Euler equations
\begin{align*}
 \del_t \rho + \div (\rho \vec v) &= 0\\
 \del_t (\rho \vec v) + \div (\rho \vec v \otimes \vec v + p \id) &= 0
\end{align*}
Here the pressure is a function of the density and is taken as $p(\rho) = K \rho^\gamma$ ($K>1$, $\gamma \geq1$). Linearization around the state $(\rho, \vec v) = (\bar \rho, 0)$ yields
\begin{align}
 \del_t \rho + \bar \rho \, \div \vec v &= 0 \label{eq:lineuler1rho}\\
 \del_t \vec v + c^2 \frac{\grad \rho}{\bar \rho} &= 0 \label{eq:lineuler1p}
\end{align}
where one defines $c = \sqrt{p'(\bar \rho)}$. Linearization with respect to a fluid state moving at some constant speed $\vec U$ can be easily removed or added via a Galilei transform.

The same system can be obtained from the Euler equations endowed with an energy equation
\begin{align*}
 \del_t \rho + \div (\rho \vec v) &= 0\\
 \del_t (\rho \vec v) + \div (\rho \vec v \otimes \vec v + p \id) &= 0\\
 \del_t e + \div(\vec v(e+p)) &= 0 
\end{align*}
with the total energy density $e = \frac{p}{\gamma-1} + \frac12 \rho |\vec v|^2$, $\gamma > 1$ which closes the system. Linearization around $(\rho, \vec v, p) = (\bar \rho, 0, \bar p)$ yields
\begin{align}
 \del_t \rho + \bar \rho\, \div \vec v &= 0 \label{eq:lineuler2rho}\\
 \del_t \vec v + \frac{\grad p}{\bar \rho} &= 0 \label{eq:lineuler2v}\\
 \del_t p + \bar \rho c^2 \,\div \vec  v &= 0 \label{eq:lineuler2p}
\end{align}
Equations \eqref{eq:lineuler2v}--\eqref{eq:lineuler2p} are (up to rescaling and renaming) the same as \eqref{eq:lineuler1rho}--\eqref{eq:lineuler1p}. Both can be linearly transformed to the symmetric version
\begin{align}
 \del_t \vec v + c \,\grad p &= 0 \label{eq:acousticv}\\
 \del_t p + c\,\div \vec v &= 0 \label{eq:acousticp}
\end{align}
$p$ will be called pressure and $\vec v$ the velocity -- just to have names. Due to the different linearizations and the symmetrization they are not exactly the physical pressure or velocity any more, but still closely related. These equations describe the time evolution of small perturbations to a constant state of the fluid.

It is to be noted that system \eqref{eq:acousticv}--\eqref{eq:acousticp} does not in general reduce to the usual wave equation, and thus does not admit the usual Kirchhoff solution (\cite{evans98}). The equation for the scalar $p$ is indeed the usual scalar wave equation
\begin{align}
 \del_t^2 p - c^2 \Delta p &= 0 \label{eq:acousticwavep}
\end{align}
but $\vec v$ fulfills
\begin{align}
 \del_t^2 \vec v - c^2 \grad \div \vec v &= 0 \label{eq:acousticwavev}
\end{align}
The identity $\nabla \times ( \nabla \times \vec v) = \nabla(\nabla \cdot v) - \Delta \vec v$ links this operator to the vector Laplacian in 3-d. By \eqref{eq:acousticv} 
\begin{align}
\del_t (\nabla \times \vec v) = 0 \label{eq:vorticity}
\end{align}
but $\nabla \times \vec v$ needs not be zero initially. Equation \eqref{eq:acousticwavev} cannot be split into scalar wave equations for the components. This is why the solution to linear acoustics is more complicated than that of a scalar wave equation. Equation \eqref{eq:acousticwavev} so far has not been given much attention in the literature. However, its behaviour differs from that of the scalar wave equation, which manifests itself, for example, in the occurrence of a vorticity singularity in the solution of a multi-dimensional Riemann Problem (as observed in \cite{amadori2015}). This is a feature of the particular vector wave equation \eqref{eq:acousticwavev} only. This singularity is studied in more detail in Section \ref{sec:riemann}.
 
\subsection{Low Mach number limit}

 The system \eqref{eq:acousticv}--\eqref{eq:acousticp} has a low Mach number limit just as the Euler equations (compare \cite{dellacherierieper10} and for the Euler case see e.g. \cite{klainerman81,klein95,metivier01}). One introduces a small parameter $\epsilon$ and inserts the scaling $\epsilon^{-2}$ in front of the pressure gradient in \eqref{eq:lineuler2v} such that the system and its symmetrized version read, respectively,
\begin{align}
 \del_t \vec v + \frac{1}{\epsilon^2} \,\grad p &= 0 \label{eq:acousticscaledv}\\
 \del_t p + c^2\,\div \vec v &= 0 \label{eq:acousticscaledp}
\end{align}
and
\begin{align}
 \del_t \vec v + \frac{c}{\epsilon} \,\grad p &= 0 \label{eq:acousticscaledvsym}\\
 \del_t p + \frac{c}{\epsilon}\,\div \vec v &= 0 \label{eq:acousticscaledpsym}
\end{align}
 In one spatial dimension the transformation which symmetrizes the Jacobian $J = \left( \begin{array}{cc}  0 & \frac{1}{\epsilon^2}\\ c^2 & 0 \end{array}\right )$ is
 \begin{align}
  S = \left( \begin{array}{cc}  1 & \\& c \epsilon   \end{array} \right ) \label{eq:trafomatrix}
 \end{align}
 such that $J = S \left( \begin{array}{cc} 0&\frac{c}{\epsilon}\\ \frac{c}{\epsilon}&0  \end{array}\right ) S^{-1}$. In multiple spatial dimensions the upper left entry in $S$ has to be replaced by an appropriate block-identity-matrix. In the following preference is given to the symmetric version if nothing else is stated. Regarding the low Mach number limit the non-symmetrized version is more natural and will be reintroduced for studying the low Mach number properties of the scheme. The variables of the two systems are given the same names to simplify notation.

\section{Exact solution}
\label{sec:solution}

Consider the Cauchy problem for the multi-dimensional linear hyperbolic system
\begin{align}
 \del_t q + (\vec J \cdot \nabla) q &= 0 \label{eq:system} & q : \mathbb R^+_0 \times \mathbb R^d \to \mathbb R^m\\
 q(0, \vec x) &= q_0(\vec x)
\end{align}
where $\vec x \in \mathbb R^d$ and $\vec J$ is the vector of the Jacobians into the different directions\footnote{Only vectors with $d$ components are typeset in boldface letters. Indices never denote derivatives.} {:
\begin{align*}
 (\vec J \cdot \nabla) q = J_x \del_x q + J_y \del_y q + J_z \del_z q
\end{align*}
}  and $m$ is the size of the system. For the symmetrized system \eqref{eq:acousticv}--\eqref{eq:acousticp} in 3-d one has $q := (\vec v, p)$ such that $m=d+1$ and 
\begin{align}
 \vec J = \left( \left( \begin{array}{cccc} 0\,\, &&&c\,\,\\&0\,\,\\&&0\,\,\\c\,\, &&&0\,\,\end{array} \right) ,
      \left( \begin{array}{cccc} 0\,\,\\&0\,\,&&c\,\,\\&&0\,\,\\&c\,\, &&0\,\, \end{array} \right),
     \left( \begin{array}{cccc} 0\,\,\\&0\,\,\\&&0\,\,&c\,\,\\&&c\,\,&0\,\,\end{array} \right)
     \right ) \label{eq:acousticjacobians}
\end{align}

The solution formula for \eqref{eq:acousticv}--\eqref{eq:acousticp} turns out to contain \emph{derivatives} of the initial data (Section \ref{ssec:solution}). This makes it necessary to consider \emph{distributional solutions}, in order to be able to differentiate a jump. A brief review of distributions is given in Section \ref{ssec:distributions} of the Appendix. This is needed in order to fix the notation used. For the sake of better readability, the derivation of the distributional solution of linear acoustics is performed in Section \ref{sec:distributionderivation} of the Appendix, whereas Section \ref{ssec:solution} only states the resulting solution formulae.

\subsection{Solution formulae for the multi-dimensional case} \label{ssec:solution}

The following notation is used throughout: Choosing $r \in \mathbb R^+$ and $d \in \mathbb N^+$ the $d$-ball of radius $r$ is denoted by
 $B^{d}_r := \{ \vec x \in \mathbb R^d : |\vec x| \leq r \}$ and let the sphere $S_r^{d-1}$ denote its boundary.

\begin{definition}[Evolution operator] \label{def:evolutionoperator}
The \emph{evolution operator} $T_t$ maps suitable initial data $q_0(\vec x)$ to the solution of the corresponding Cauchy problem for \eqref{eq:system} (that is assumed to exist and be unique) at time $t$:
\begin{align*}(T_t\, q_0)(t, \vec x) = q(t, \vec x)\end{align*}
\end{definition}
Obviously $T_0 = \textrm{id}$.

{The distributional solution cannot be stated prior to introducing corresponding notation. Therefore the full Theorem is stated as Theorem \ref{thm:solutionappendix} and proven in the Appendix. Below only the case of smooth initial data is shown to simplify the presentation.}

\begin{theorem}[Solution formulae] \label{thm:solution}
 If {$p_0 \in C^2 \cap L^\infty(\mathbb R^3)$ and $\vec v_0$ a $C^2 \cap L^\infty$ vector field}, then 
 \begin{align}
   p(t, \vec x) &= p_0(\vec x) + \int_0^{ct} \dd r \, r \frac{1}{4\pi} \int_{S^2_1} \dd {\vec y} \,\,(\div \grad p_0)(\vec x + r \vec y) - ct \frac{1}{4\pi}\int_{S^2_1} \dd {\vec y}\,\, \div \vec v_0 (\vec x + ct \vec y)\label{eq:phelmholtzfunc}\\
 \vec v(t, \vec x) &= \vec v_0(\vec x) + \int_0^{ct} \dd r \, r \frac{1}{4\pi}\int_{S^2_1} \dd {\vec y} \,\, (\grad \div \vec v_0)(\vec x + r \vec y) - ct \frac{1}{4\pi}\int_{S^2_1} \dd {\vec y} \,\, (\grad p_0)(\vec x + ct \vec y)\label{eq:vhelmholtzfunc}
  \end{align}
  is solution to
  \begin{align*}
    \del_t q + \vec J \cdot \nabla q &= 0 & p(0, \vec x) &= p_0(\vec x) & \vec v(0, \vec x) &= \vec v_0(\vec x)
  \end{align*}
  with $\vec J$ given by \eqref{eq:acousticjacobians}, and with $d=3$.
\end{theorem}

{The proof of this Theorem is given in Section \ref{sec:distributionderivation} of the Appendix (Theorem \ref{thm:solutionappendix} and Corollary \ref{cor:helmholtzappendix} therein). One observes the appearance of a \emph{spherical average}
\begin{align*}
 \frac{1}{4\pi}\int_{S^2_1} \dd {\vec y} \,\, f(\vec x + r \vec y)
\end{align*}
of a function $f(\vec x)$. The distributional analogue can be found in Definition \ref{def:sphericalaverage}. }

Spherical means appear already in the study of the scalar wave equation (see e.g. \cite{john78,evans98}). As has been discussed in Section \ref{ssec:equationslinearization}, only the evolution of the scalar variable $p$ is governed by a scalar wave equation. The equation for $v$ is a vector wave equation. However, usage of the Helmholtz decomposition allows to write down a scalar wave equation for the \emph{curl-free part} of $\vec v$, whereas the time evolution of the curl is given by $\del_t (\nabla \times \vec v ) = 0$. The solution to the scalar wave equation can be used and the Helmholtz decomposition of the two parts of the velocity conveniently reassembles into \eqref{eq:phelmholtzfunc}--\eqref{eq:vhelmholtzfunc}. The above formulae appear without proof in \cite{eymann13} where they have been obtained by this analogy with the scalar wave equation. A similar approach is taken in \cite{franck17}, again assuming that the solution is smooth enough. It is important to note however, that the initial data onto which the solution formulae are applied in \cite{eymann13} are not sufficiently well-behaved for the second derivatives to exist, such that a justification in the sense of distributions was needed. This is now accomplished in Theorem \ref{thm:solutionappendix}. {However, the notational overhead of distribution theory might obscure interesting properties of the solution operator that are discussed below. Although the formulae (upon reinterpretation) remain valid for the distributional solution, the properties are presented using Theorem \ref{thm:solution}. The distributional solution is presented as a self-contained Section \ref{sec:distributionderivation}.}

From the solution in Equation \eqref{eq:vhelmholtzfunc} one observes that $\vec v$ changes in time by a gradient of a function. Applying the curl, this gradient vanishes -- indeed, the curl must be stationary due to Eq. \eqref{eq:vorticity}.

The spatial derivatives that appear in the solution formulae \eqref{eq:phelmholtzfunc}--\eqref{eq:vhelmholtzfunc} can be transformed into derivatives with respect to $r$ only. The new formulae are more useful in certain situations (as will be seen later), and display interesting properties of the solution that are discussed after stating the Theorem.

\begin{corollary}[Solution formulae with radial derivatives only]  \label{cor:solutionradial}

  Consider the setup of Theorem \ref{thm:solution}. {For $C^2 \cap L^\infty$ initial data $p_0, \vec v_0$} the solution \eqref{eq:phelmholtzfunc}--\eqref{eq:vhelmholtzfunc} can be rewritten as
\begin{align}
 p(t, \vec x) &= \del_r \left(r \frac{1}{4\pi}\int_{S^2_1} \dd {\vec y} \,    p_0 \right) - \frac{1}{r} \del_r \left(r^2 \frac{1}{4\pi}\int_{S^2_1} \dd {\vec y} \,  \vec y \cdot \vec v_0\right ) \label{eq:preformulfunc}\\
 \vec v(t, \vec x) &= \frac23 \vec{ v}_0(\vec x) - \frac{1}{r} \del_r \left( r^2 \frac{1}{4\pi}\int_{S^2_1} \dd {\vec y} \, p_0  \vec y \right) + \del_r\left(r\frac{1}{4\pi}\int_{S^2_1} \dd {\vec y} \,  (\vec v_0 \cdot \vec y)  \vec y \right ) \nonumber \\&- \frac{1}{4\pi}\int_{S^2_1} \dd {\vec y}\,  \left[  \vec v_0 - 3  (\vec v_0 \cdot \vec y)  \vec y   \right ] - \int_0^{ct} \dd r \frac{1}{r} \frac{1}{4\pi}\int_{S^2_1} \dd {\vec y}\,  \left[  \vec v_0 - 3  (\vec v_0 \cdot \vec y)  \vec y   \right ] \label{eq:vsolution23func}
\end{align}
and Equation \eqref{eq:vsolution23func} is equivalent to
\begin{align}
 \vec v(t, \vec x) = \vec v_0(\vec x) &- \frac{1}{r}\del_r \left( r^2 \frac{1}{4\pi}\int_{S^2_1} \dd {\vec y} \,  p_0  \vec y \right) \nonumber\\&+ \int_0^{ct} \dd r \frac{1}{r} \del_r \left[ \frac{1}{r} \del_r \left( r^3 \frac{1}{4\pi}\int_{S^2_1} \dd {\vec y}   (\vec v_0 \cdot \vec y) \vec y \right )  -  r \frac{1}{4\pi}\int_{S^2_1} \dd {\vec y} \vec v_0   \right ] \label{eq:ureformulfunc}
\end{align}

\end{corollary}
{\sl Note}: Everything (if not stated explicitly) is understood to be evaluated at $\vec x + r \vec y$, and wherever $r$ remains, $r=ct$ is to be taken at the very end. 
For example, the term $\del_r \left(r \frac{1}{4\pi}\int_{S^2_1} \dd {\vec y} \,    p_0 \right)$ appearing in \eqref{eq:preformulfunc}, if written explicitly, reads
\begin{align*}
 \left. \del_r \left(r \frac{1}{4\pi}\int_{S^2_1} \dd {\vec y} \,    p_0(\vec x + r \vec y) \right) \right |_{r=ct}
\end{align*}

This Corollary follows from its distributional counterpart Theorem \ref{thm:solutionradialappendix}, proven in the Appendix. Therein it is in particular shown that the integral
\begin{align*}
 \int_0^{ct} \dd r \frac{1}{r} \int_{S^2_1} \dd {\vec y}\,  \left[  \vec v_0 - 3  (\vec v_0 \cdot \vec y)  \vec y   \right ]
\end{align*}
is finite for continuous $\vec v_0$. {For $C^2$ data \eqref{eq:preformulfunc}--\eqref{eq:ureformulfunc} can be shown to follow directly from \eqref{eq:phelmholtzfunc}--\eqref{eq:vhelmholtzfunc} by Gauss' theorem for the sphere of radius $r$. For example, differentiating
\begin{align*}
 \del_r \left(r \frac{1}{4\pi}\int_{S_1^2} \dd \vec y \,    p_0 (\vec x + r\vec y) \right)
\end{align*}
with respect to $r$ yields
\begin{align*}
 \del_r^2 \left(r \frac{1}{4\pi}\int_{S_1^2} \dd \vec y \,     p_0(\vec x + r\vec y) \right) &= \frac{1}{r} \del_r \left( r^2 \del_r \frac{1}{4\pi}\int_{S_1^2} \dd \vec y \,  p_0(\vec x + r\vec y) \right )
\intertext{by elementary manipulations. Differentiation with respect to $r$ can be replaced by $\vec y \cdot \nabla$ inside the spherical mean:}
 \frac{1}{r} \del_r \left( r^2 \del_r \frac{1}{4\pi}\int_{S_1^2} \dd \vec y \,  p_0(\vec x + r\vec y) \right )&= \frac{1}{r} \del_r \left( r^2 \frac{1}{4\pi}\int_{S_1^2} \dd \vec y \,  \vec y \cdot \nabla p_0(\vec x + r\vec y) \right )\\
\intertext{and by Gauss theorem}
 &= \frac{1}{r} \del_r \int_0^r \dd r' \left( r'^2 \frac{1}{4\pi}\int_{S_1^2} \dd \vec y \,  \nabla \cdot \nabla p_0(\vec x + r'\vec y) \right )\\ 
 &= r \frac{1}{4\pi}\int_{S_1^2} \dd \vec y \,  \nabla \cdot \nabla p_0(\vec x + r'\vec y) 
\end{align*}
Integrating over $r$, and evaluating at $r=ct$ yields the sought identity
\begin{align}
 \left. \del_r \left( r \frac{1}{4\pi}\int_{S_1^2} \dd \vec y \,   p_0 (\vec x + r\vec y) \right ) \right |_{r=ct} &= p_0(\vec x) + \int_0^{ct} \dd r \,r \frac{1}{4\pi}\int_{S_1^2} \dd \vec y \,  \nabla \cdot \nabla p_0(\vec x + r\vec y)  \label{eq:derivativetransfer}
\end{align}
In a similar but slightly more complicated way the equivalence of the other terms can be shown.

}

\subsection{Properties of the solution} \label{ssec:properties}

There is a number of striking differences to the one-dimensional case that appear in multiple spatial dimensions. 

For the one-dimensional problem only the values of the initial data appear in the solution formulae, not their derivatives. This is different in multiple dimensions and can already be observed for the scalar wave equation (as discussed e.g. in \cite{evans98}). A similar statement is true for the solution to Equations \eqref{eq:acousticv}--\eqref{eq:acousticp}. (As explained in Section \ref{ssec:equationslinearization} this system cannot be reduced to scalar wave equations.) \eqref{eq:phelmholtzfunc}--\eqref{eq:vhelmholtzfunc} make the impression that second derivatives of the initial data need to be computed, but Corollary \ref{cor:solutionradial} states that the solution can be rewritten into Equation \eqref{eq:vsolution23func}, which involves only first spatial derivatives.

In one spatial dimension, the solution at a point $x$ depends only on initial data at points $y$ for which $|y-x|= ct$. This motivates the following (compare e.g. \cite{oneill1983}, Chapter 14)
\begin{definition}[Causal structure]
 Let $(t, \vec x) \in \mathbb R^+_0 \times \mathbb R^d$. The restriction of the initial data onto the set 
 $$\mathscr T_{(t, \vec x)}:=\{ \vec y : |\vec x - \vec y| < ct \}$$ is called \emph{timelike initial data} for $(t,\vec x)$. The restriction of the initial data onto the set 
 $$\mathscr N_{(t, \vec x)}:=\{ \vec y : |\vec x - \vec y| = ct \}$$ is called \emph{null initial data} for $(t,\vec x)$.
\end{definition}

In one spatial dimension the solution to the acoustic equations depends on null initial data only. In multiple spatial dimensions the situation is more complicated. The solution of just the scalar wave equation \eqref{eq:acousticwavep} depends on null initial data for odd dimensions $d=1, 3, 5, \ldots$, whereas in even spatial dimensions $d=2, 4, \ldots$ it also depends on timelike initial data (see e.g. \cite{john78,evans98}). For the acoustic system \eqref{eq:acousticv}--\eqref{eq:acousticp}, which involves a scalar as well as a vector wave equation \eqref{eq:acousticwavev}, the solution depends on timelike initial data \emph{both} in odd and even number of spatial dimensions. 

\subsection{Example of a singularity in the two-dimensional Riemann problem}
\label{sec:riemann}

In this Section, the exact solution is applied to a two-dimensional Riemann problem. {As the solution formulae are applied to discontinuous initial data, here Theorem \ref{thm:solution} and {Corollary} \ref{cor:solutionradial} are not sufficient, and the distributional versions \ref{thm:solutionappendix}, \ref{thm:solutionradialappendix} need to be used. Therefore in this section, notation and results from Sections \ref{ssec:distributions} and \ref{sec:distributionderivation} are used. The reader might thus want to consult them first. The result of the computation is Equation \eqref{eq:vysolutionriemann}.}

It shall be shown in the following how a multi-dimensional Riemann Problem can exhibit a logarithmic singularity in its evolution. This is due to the fact that the acoustic system does not reduce to scalar wave equations, but also contains the vector wave equation \eqref{eq:acousticwavev}.

For computational convenience the Riemann Problem is considered in the $x$-$z$-plane. The initial velocity shall be $\vec v_0 = (0,0,1)^\text{T}$ in the first quadrant and vanish everywhere else (see Fig. \ref{fig:initialdata}). Also everywhere $p_0=0$. 

\begin{figure}[h]
 \centering
 \includegraphics[width=0.4\textwidth]{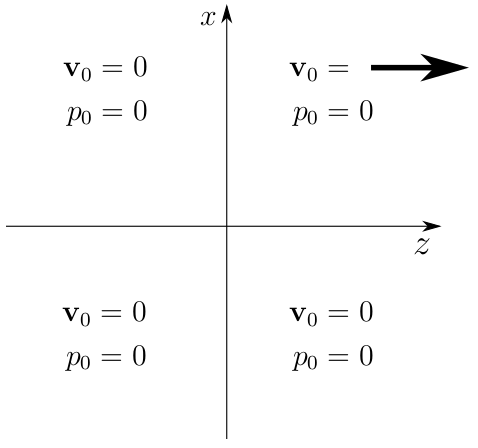}
 \caption{Setup of the 2-dimensional Riemann Problem. The only non-vanishing initial datum is the $z$-velocity in the first quadrant, indicated by the arrow. As the problem is linear its magnitude is of no importance and is chosen to be 1.}
 \label{fig:initialdata}
\end{figure}

 Denote the independent variable $\vec x =: (x, y, z)$ and the components of $\vec v =: (v_x, v_y, v_z)$, $\vec v_0 =: (v_{0x}, v_{0y}, v_{0z})$.

 The distribution $\sigma_{ij}$ defined in Theorem \ref{def:sigmaij} acts onto test functions as
 \begin{align*}
 \langle \sigma_{ij}(ct) | \test \rangle := \int_0^{ct} \dd r \frac{1}{r} \del_r \left[ \frac{1}{r} \del_r \left( r \int_{S^2_r} \dd \vec y \frac{y_i y_j}{|\vec y|^2} \psi(\vec y) \right )  -  \frac{1}{r} \int_{S^2_r} \dd \vec y \psi(\vec y)   \right ]  
 \end{align*}
 {Its components are denoted by $\sigma_{xx}, \sigma_{xy}, \ldots, \sigma_{zz}$.}
 
 Define the components of $\vec y =: (y_x, y_y, y_z)$. {The Riemann initial data are conveniently written as} $v_{0z}(\vec x) = \Theta(x) \Theta(z)$ {with the Heaviside function
 \begin{align}
  \Theta(x) = \begin{cases} 0 & x < 0 \\ 1 & x \geq 0 \end{cases}
 \end{align}}

 Inserting this and $v_{0x} = v_{0y} = 0$ into \eqref{eq:ureformulappendix} gives
 
 \begin{align*}
  \langle v_x(t, \cdot) | \test \rangle  &=  \frac{1}{4\pi} \langle {\sigma_{zx}(ct)} * \regdist{v_{0z}} | \test \rangle  \\
 &=  \frac{1}{4\pi} \int_0^{ct} \dd r \frac{1}{r} \del_r \left[ \frac{1}{r} \del_r \left( r \int_{S^2_r} \dd \vec y \frac{y_x y_z}{|\vec y|^2} \int \dd \vec x \, \Theta(x) \Theta(z) \test(\vec x + \vec y) \right )     \right ] 
 \end{align*}
 Compute first
 \begin{align*}
  \int_{S^2_r} \dd \vec y \frac{y_x y_z}{|\vec y|^2} \int \dd \vec x \, \Theta(x) \Theta(z) \test(\vec x + \vec y) 
  &= \int \dd \vec x \int_{S^2_r} \dd \vec y \frac{y_x y_z}{r^2}  \Theta(x-y_x) \Theta(z-y_z) \test(\vec x)
 \end{align*}
 This defines a regular distribution associated to
 \begin{align*}
 &\int \dd \vec y \frac{y_x y_z}{r^2}  \Theta(x-y_x) \Theta(z-y_z)\\
 \intertext{Evaluating the integral for the special case of $x = 0$ one obtains}
 &=\int_{-r}^{\min(r, z)} \dd y_z \int_{-\sqrt{r^2 - y_z^2}}^{0} \dd y_x \frac{2 y_x y_z}{r^2 \sqrt{1 - y_x^2 - y_z^2}} = \frac{2 (r^2 - z^2)^{\frac32}}{3 r} \Theta(r - |z|)
 \end{align*}
 Here the first fundamental form of the unit sphere $(1 - y_x^2 - y_z^2)^{-\frac12}$ was used to express the surface integral.

 The velocity becomes 
 \begin{align*}
  v_x(t, \vec x)  &= \frac{1}{4\pi} \int_{z}^{ct} \dd r \frac{1}{r} \del_r \left[ \frac{1}{r} \del_r \left( r \frac{2 (r^2 - z^2)^{3/2}}{3 r} \right ) \right ]
 \end{align*}
 and using that for any function $f$
 \begin{align*}
 \frac{1}{r} \del_r f(r^2) = 2 f'(r^2)
 \end{align*}
 one obtains
 \begin{align}
  v_x(t, \vec x)  &= \frac{1}{2\pi} \int_{z}^{ct} \dd r   (r^2 - z^2)^{-1/2} = \frac{1}{2 \pi} \mathscr L\left( \frac{z}{ct}  \right ) \label{eq:vysolutionriemann}
 \end{align}
  
having defined
\begin{align*}
 \mathscr L(s) := \ln \frac{1 + \sqrt{1 - s^2}}{s} = - \ln \frac{s}{2} - \frac{s^2}{4} + \mathcal O(s^4)
\end{align*}
One can verify that $\displaystyle \text{e}^{- \mathscr L(s)} = \tan \frac{\arcsin s}{2}$.

Note that due to the appearance of the factor $\Theta(r - |z|)$ above, $v_x(t, \vec x)$ vanishes outside $|\vec x| \leq ct$ by causality.

Therefore the $x$-component of the velocity has a logarithmic singularity at the origin, which is the corner of the initial discontinuity of the $z$-component. Such a behaviour of the solution does not have analogues in one spatial dimension because two different components of the velocity $\vec v$ are involved. The appearance of singularities has already been mentioned in \cite{amadori2015} in the context of self-similar solutions to Riemann Problems. Here it has been obtained by application of the general formula \eqref{eq:preformulappendix}--\eqref{eq:ureformulappendix} which is not restricted to self-similar time evolution. Moreover a careful derivation using distributional solutions, adequate for the low regularity of the initial data, has been presented.

The solution obtained so far was restricted to $x=0$ to simplify the presentation. This was also sufficient in order to study the appearance of a singularity. For Fig. \ref{fig:velxNE}--\ref{fig:vectors} the integrals in \eqref{eq:ureformulappendix} have been computed in the $x$-$z$-plane numerically using standard quadratures. They give an impression of the entire solution of the two-dimensional Riemann Problem. It is not very difficult to obtain analytic expressions in all the plane by slightly adapting the above calculations.

A vector plot of the flow is shown in Fig. \ref{fig:vectors}.

\begin{figure}[h]
 \centering
\includegraphics[width=0.32\textwidth]{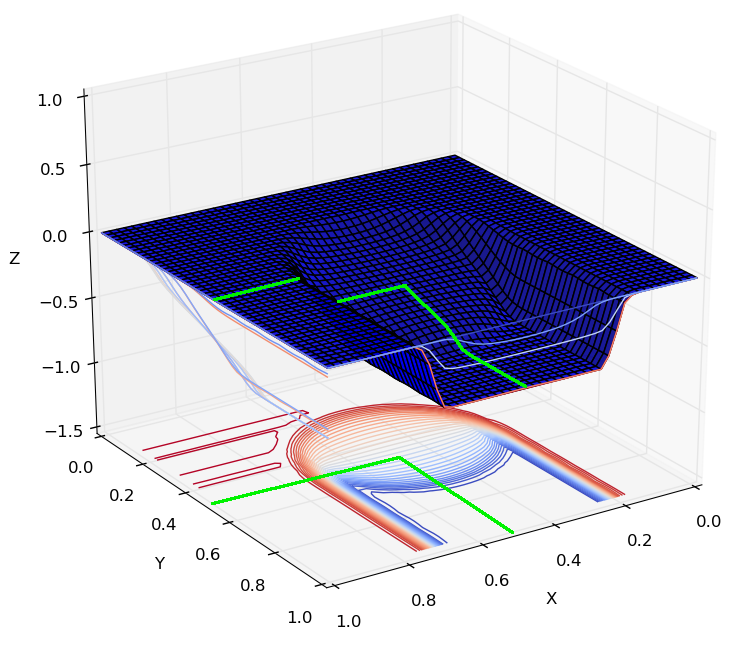}
\includegraphics[width=0.32\textwidth]{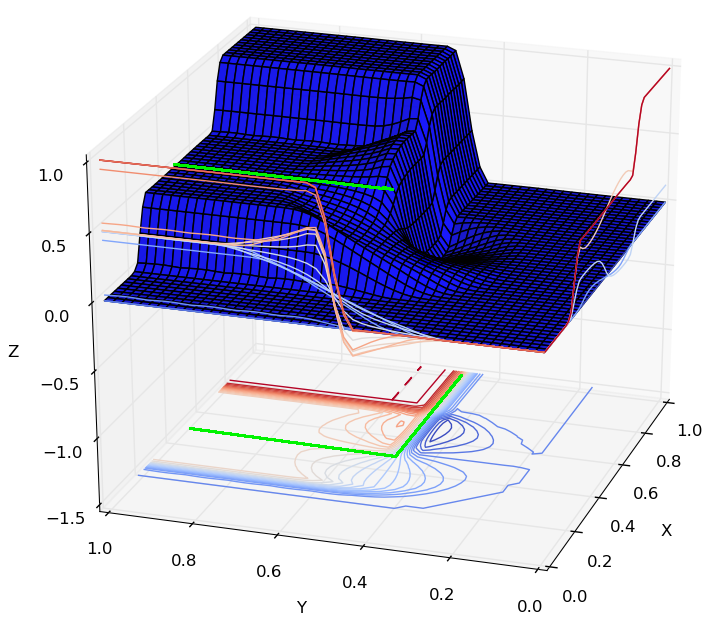}
\includegraphics[width=0.32\textwidth]{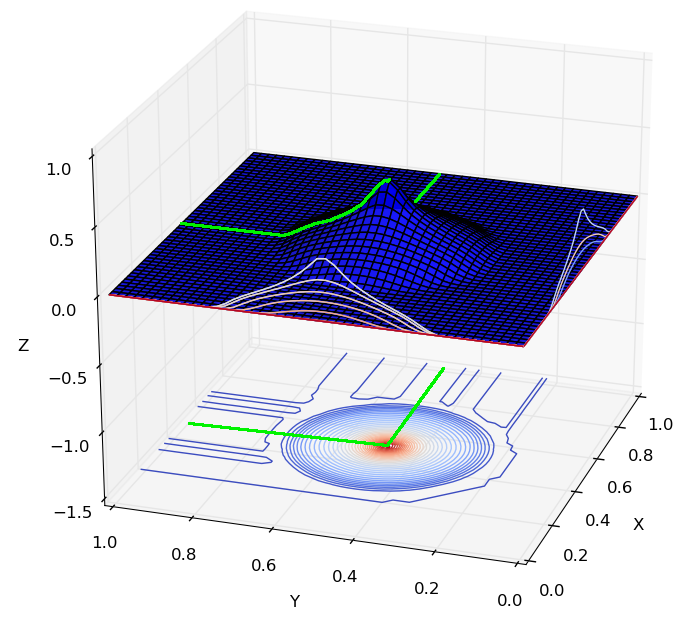}
\caption{Solution of Riemann problem at time $ct = 0.25$. {\sl Left}: Pressure. {\sl Center}: $x$-velocity. {\sl Right}: $y$-velocity. The smoothed out discontinuities are due to finite size sampling of the solution. In green the location of the initial discontinuity.}
\label{fig:velxNE}
\end{figure}

\begin{figure}[h]
 \centering
 \includegraphics[width=0.6\textwidth]{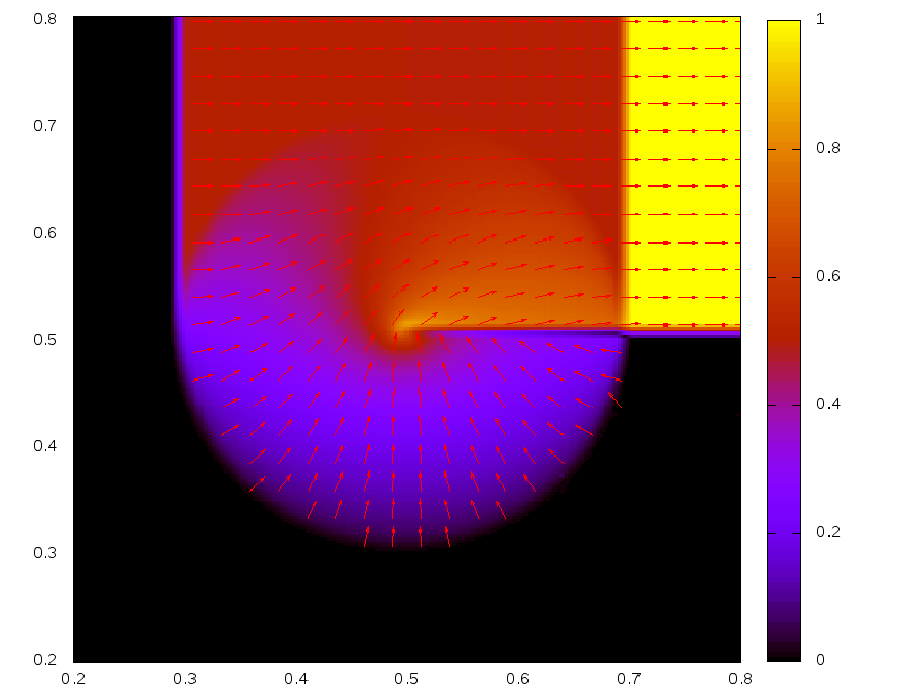}
 \caption{Solution of Riemann problem at time $ct = 0.2$. The direction of the velocity $\vec v(t, \vec x)$ is indicated by the arrows, colour coded is the absolute value $|\vec v|$.}
 \label{fig:vectors}
\end{figure}

\section{Godunov finite volume scheme}
\label{sec:godunov}

The exact evolution operator for linear acoustics \eqref{eq:acousticp}--\eqref{eq:acousticv} already appears in \cite{eymann13,roe17}, albeit without the justification as distributional solution. It has been taken as inspiration for a new kind of numerical schemes in \cite{eymann13}: the \emph{active flux} method contains additional, pointwise degrees of freedom which are evolved in time exactly. A finite volume scheme of the usual kind was derived in \cite{franck17}, but has taken only an approximation of this exact solution into account. Among others, \cite{lukacova00} considered the bicharacteristics relation (\cite{courant62}) in order to derive schemes which incorporate multi-dimensional information. However, the bicharacteristics involve mantle integrals along the characteristic cone and do not allow to write down the solution as a function of initial data directly. Thus again, these schemes only use an approximation of the exact relation.

The conceptually simplest finite volume is a Godunov scheme with the Riemann Problem as a building block, \cite{abgrall93,li02,amadori2015} studied the solutions to multi-dimensional Riemann Problems for linear acoustics using a self-similarity ansatz. However, no numerical scheme has been derived there. 

This paper presents a full derivation of such a scheme and discusses numerical results obtained with it. This is similar in spirit to an idea by Gelfand mentioned in \cite{godunov76,godunov97} (a translated version is \cite{godunov08}).

\subsection{Procedure} \label{ssec:godunovprocedure}

In this Section the aim is to derive a two-dimensional finite volume scheme, which updates the numerical solution $q_{ij}^n$ in a Cartesian cell $\mathcal C_{ij} = [x_{i-\frac12}, x_{i+\frac12}] \times [y_{j-\frac{1}{2}}, y_{j+\frac12}]$ at a time $t^n$ to a new solution $q_{ij}^{n+1}$ at time $t^{n+1} = t^n + \Delta t$ using $q_{ij}^n$ and information from the neighbours of $\mathcal C_{ij}$. The grid is taken equidistant. 

The knowledge of the exact solution makes it possible to derive a Godunov scheme via the procedure of \emph{reconstruction}-\emph{evolution}-\emph{averaging} (\cite{toro09,leveque02}). In the following it is shown that for linear systems this can be achieved using just \emph{one} evaluation of the solution formula at a single point in space by suitably modifying the initial data.

Consider the general linear $n \times n$ hyperbolic system \eqref{eq:system} in $d$ spatial dimensions
\begin{align*}
 \del_t q + (\vec J \cdot \nabla) q &= 0 
\end{align*}
with initial data
\begin{align*}
 q(0, \vec x) = q_0(\vec x)
\end{align*}

Recall the Definition \ref{def:evolutionoperator} of the time evolution operator $T_t$: $(T_t\, q_0)(t, \vec x)$ satisfies \eqref{eq:system} with (at $t=0$) initial data $q_0(\vec x)$.

\begin{definition}[Sliding average]
 Define the sliding average operator $A$ in two spatial dimensions by its action onto a function $q : \mathbb R^d \to \mathbb R^n$ as
\begin{align}
 (A q)(\vec x) := \frac{1}{\Delta x \Delta y} \!\!\!\!\!\!\!\!\!\!\! \int\limits_{\left[-\frac{\Delta x}{2},\frac{\Delta x}{2}\right] \times \left[-\frac{\Delta y}{2}, \frac{\Delta y}{2}\right]} \!\!\!\!\!\!\!\!\!\! \dd \vec s \,\,\,\,\,\,\,\,\,\, q(\vec x + \vec s) \label{eq:slidingaverage}
\end{align}
\end{definition}

The objective is to construct a Godunov scheme by introducing a reconstruction {$q^n_\text{recon}(\vec x)$ obtained from} the discrete values $\{ q_{ij}^n \}$ in the cells and computing its exact time evolution. The reconstruction needs to be conservative, i.e. $(A q^n_\text{recon})(\vec x_{ij}) = q^n_{ij}$. The easiest choice is a piecewise constant reconstruction
\begin{align*}
 q^n_\text{recon}(\vec x) := q^n_{ij} \qquad \text{if} \qquad \vec x \in \mathcal C_{ij}
\end{align*} 
It is shown in Fig. \ref{fig:rea} (left) and obviously is locally integrable. 

The Godunov procedure \emph{reconstruction}-\emph{evolution}-\emph{averaging} can be written as
\begin{align*}
 q_{ij}^{n+1} = (A\, T_{\Delta t}\, q^n_\text{recon})(\vec x_{ij})
\end{align*}

\begin{figure}[h]
 \centering
 \includegraphics[width=0.48\textwidth]{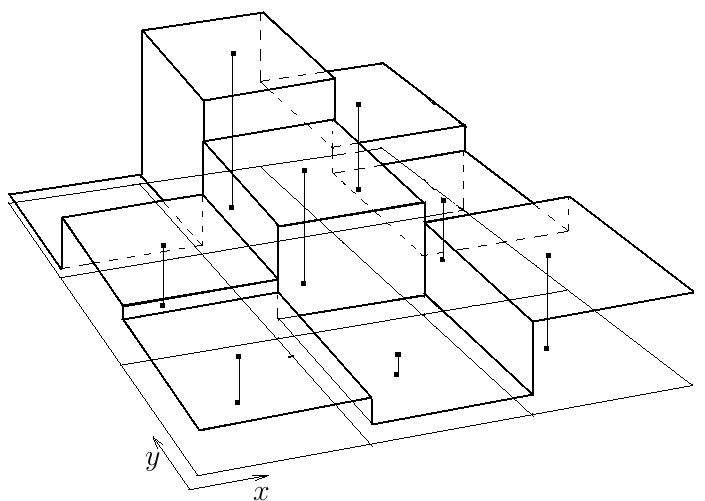} \hfill \includegraphics[width=0.48\textwidth]{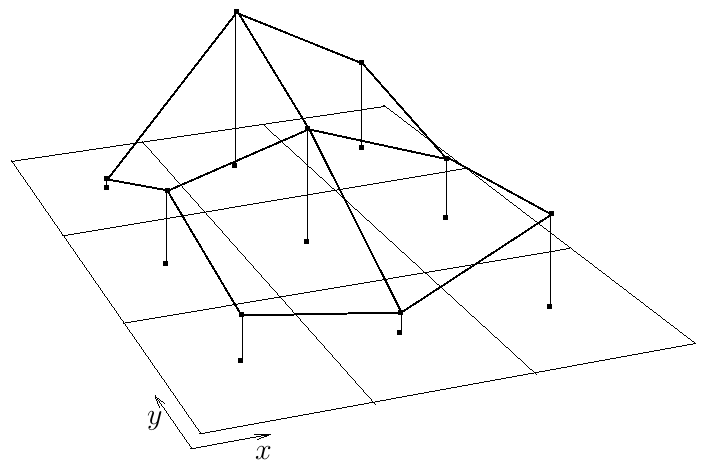}
 \caption{ {\sl Left}: Piecewise constant reconstruction. {\sl Right}: Application of the sliding average to the same data amounts to a bilinear interpolation of the values $q_{ij}$ interpreted as \emph{point values} at $\vec x_{ij}$. The two reconstructions are equivalent for linear problems, i.e. they lead to the same Godunov scheme.}
 \label{fig:rea}
\end{figure}

\begin{lemma} Provided all expressions exist, the two operators commute:
\begin{align*}
 \left. \phantom{\frac{x}{x}} A\, T_{\Delta t}\, q^n_\text{recon} \right |_{\vec x_{ij}} = \left. \phantom{\frac{x}{x}} T_{\Delta t}\, A \,q^n_\text{recon} \right |_{\vec x_{ij}}
\end{align*}
\end{lemma}
\begin{proof} 
By linearity of $T_{\Delta t}$ 
\begin{align*}
 (A \, T_{\Delta t}\, q^n_\text{recon})(\vec x) &= \frac{1}{\Delta x \Delta y} \!\!\!\!\!\!\!\!\!\!\! \int\limits_{\left[-\frac{\Delta x}{2},\frac{\Delta x}{2}\right] \times \left[-\frac{\Delta y}{2}, \frac{\Delta y}{2}\right]} \!\!\!\!\!\!\!\!\!\! \dd \vec s \,\,\,\,\,\,\,\,\,\, (T_{\Delta t}\, q^n_\text{recon})(\vec x + \vec s) \\
 &= \frac{1}{\Delta x \Delta y} \,\,\,\,\,\, T_{\Delta t} \!\!\!\!\!\!\!\!\!\!\! \int\limits_{\left[-\frac{\Delta x}{2},\frac{\Delta x}{2}\right] \times \left[-\frac{\Delta y}{2}, \frac{\Delta y}{2}\right]} \!\!\!\!\!\!\!\!\!\! \dd \vec s \,\,\,\,\,\,\,\,\,\,  q^n_\text{recon}(\vec x + \vec s)\\
 &= T_{\Delta t} \, (A q^n_\text{recon})(\vec x)
\end{align*}

\end{proof}

In short, for linear systems the last two steps of \emph{reconstruction}-\emph{evolution}-\emph{averaging} can be turned around to be \emph{reconstruction}-\emph{averaging}-\emph{evolution} which tremendously simplifies the derivation. It should be noted that in practice, evaluating the solution formulae can be technically demanding. Therefore the straightforward derivation that first obtains a solution in the entire cell requires a lot of effort. Employing the structure of the conservation law allows to rewrite the volume integral into a time integral over the boundary. Now the exact solution is only needed along the boundary of the cell, and one of the components of $\vec x$ is zero. However, one still needs to evaluate the solution formulae at a continuous set of $\vec x$ values. The strategy presented above allows to choose a coordinate system in which the solution formulae need to be computed only once at $\vec x = 0$ by taking the sliding-averaged initial data $A q^n_\text{recon}$. 

The sliding average of a piecewise constant reconstruction on a 2-d grid is a bilinear interpolation of the values $q_{ij}$, which are interpreted pointwise at locations $\vec x_{ij}$ (see Fig. \ref{fig:rea}, right). This staggered reconstruction is continuous. {It should not, however, be confused with bilinear reconstructions aimed at deriving schemes of second order: the staggered bilinear reconstruction shown in Figure \ref{fig:rea} contains exactly the same amount of information as the piecewise constant reconstruction. The reason is that between the two reconstructions the interpretation of the discrete degree of freedom changes as well. For the piecewise constant reconstruction the discrete degree of freedom is a point value. Although derived from a different viewpoint, the scheme remains conservative, of course.}

\subsection{Notation}

In order to cope with the lengthy expressions for the numerical scheme, the following notation is used:
\begin{align*}
 [q]_{i+\frac12} &:= q_{i+1} - q_i & \{q\}_{i+\frac12} &:= q_{i+1} + q_i\\
 [q]_{i\pm1} &:= q_{i+1} - q_{i-1}\\
 [[q]]_{i\pm\frac12} &:= [q]_{i+\frac12} - [q]_{i-\frac12} & \{ \{ q \}\}_{i\pm\frac12}&:= \{q\}_{i+\frac12} + \{q\}_{i-\frac12}\\
 &{= q_{i+1} - 2 q_i + q_{i-1}} & &{= q_{i+1} + 2 q_i + q_{i-1}}
\end{align*}
The only nontrivial identity is
\begin{align*}
 \{[q]\}_{i\pm\frac12} = [q]_{i+\frac12} + [q]_{i-\frac12} = [q]_{i\pm1}
\end{align*}
For multiple dimensions the notation is combined, e.g.
\begin{align*}
 [[q_i]]_{j\pm\frac12} &= q_{i,j+1} - 2 q_{ij} + q_{i,j-1} & {[[q]]_{i\pm\frac12,j}} &{= q_{i+1,j} - 2 q_{ij} + q_{i-1,j}}\\
 {\{[q]_{i+\frac12}\}_{j+\frac12} }&{= q_{i+1,j+1} - q_{i,j+1} + q_{i+1,j} - q_{ij}} &
 [[q]_{i\pm1}]_{j\pm1} &= q_{i+1,j+1} - q_{i-1,j+1} - q_{i+1, j-1} + q_{i-1,j-1}
\end{align*}

\subsection{Finite volume scheme}
\newcommand{\eps}{\epsilon}
Performing the evaluation of the exact solution formulae as outlined in Section \ref{ssec:godunovprocedure} is straightforward: In every one of the four quadrants all the derivatives of the initial data exist. At the locations where the quadrants meet, the initial data are continuous with, in general, discontinuous first derivatives. However, the derivatives are continuous in $r$, as the kinks are all oriented towards the location $\vec x_{ij}$ (see Fig. \ref{fig:rea}, right). Thus the radial derivatives never lead to the appearance of actual distributions and the solution is a function. The reason for the different behaviour as compared to the Riemann Problem in Section \ref{sec:riemann} is that here the evolution operator is applied onto sliding-averaged discontinuities which are continuous. This goes back to the \emph{averaging}-step of the Godunov procedure.

{Consider the center of the cell $(i,j)$ to be the origin of the coordinate system. The bilinear reconstruction $Aq^n_\text{recon}$ on $[0, \Delta x] \times [0, \Delta y]$ interpolates the values $q_{ij}$, $q_{i+1,j}$, $q_{i,j+1}$ and $q_{i+1,j+1}$ at the four corners:
\begin{align*}
 (Aq^n_\text{recon})(\vec x) = q^n_{ij} + \frac{[q^n]_{i+\frac12,j}}{\Delta x}x + \frac{[q^n_i]_{j+\frac12}}{\Delta y}y + \frac{[[q^n]_{i+\frac12}]_{j+\frac12}}{\Delta x \Delta y}xy \qquad \vec x \in [0, \Delta x] \times [0, \Delta y]
\end{align*}
Analogous formulae are easily obtained for the other three quadrants. This allows to compute spherical averages at $\vec x = 0$ by summing over averages in the four quadrants:
\begin{align*}
 \frac{1}{4\pi} \int_{S^1} \dd \vec y (Aq^n_\text{recon})(r \vec y) = q^n_{ij} + \left( \frac{[[q^n_i]]_{j\pm\frac12}}{4 \Delta y}  + \frac{[[q^n]]_{i\pm\frac12,j}}{4 \Delta x} \right ) r + \frac{[[ [[q^n]]_{i\pm\frac12}   ]]_{j\pm\frac12}}{6 \pi \Delta x \Delta y} r^2
\end{align*}
For linear acoustics, $q = (\vec v, p)$. Analogously the other spherical averages are obtained, for instance
\begin{align*}
 \frac{1}{4\pi} \int_{S^1} \dd \vec y \, \vec y \cdot (A\vec v^n_\text{recon})(r \vec y) = \left( \frac{[u^n]_{i\pm1,j}}{6 \Delta x} + \frac{[v^n_i]_{j\pm1}}{6 \Delta y} \right )r + \frac{ [ [[v^n]]_{i\pm\frac12}  ]_{j\pm1} + [[ [u^n]_{i\pm1} ]]_{j\pm\frac12}}{32 \Delta x \Delta y} r^2
\end{align*}
Using Equation \eqref{eq:preformulfunc} gives
\begin{align*} 
 p(t, 0) = p^n_{ij} &+ \left( \frac{[[p^n_i]]_{j\pm\frac12}}{2 \Delta y}  + \frac{[[p^n]]_{i\pm\frac12,j}}{2 \Delta x} \right ) ct + \frac{[[ [[p^n]]_{i\pm\frac12}   ]]_{j\pm\frac12}}{2 \pi \Delta x \Delta y} (ct)^2\\
 &-   
 \left( \frac{[u^n]_{i\pm1,j}}{2 \Delta x} + \frac{[v^n_i]_{j\pm1}}{2 \Delta y} \right )ct - \frac{ [ [[v^n]]_{i\pm\frac12}  ]_{j\pm1} + [[ [u^n]_{i\pm1} ]]_{j\pm\frac12}}{8 \Delta x \Delta y} (ct)^2 
\end{align*}

In order to obtain the time evolution of the velocity further spherical means are computed in complete analogy.} After carefully collecting all the different terms, and taking care of the correct $\epsilon$-scalings one obtains the following numerical scheme
\begin{align}
u^{n+1} &= u_{ij}^n - \frac{c\Delta t}{2\epsilon\Delta x} \left( [p^n]_{i\pm1, j} - [[u^n]]_{i\pm\frac12,j}    \right )  
 \nonumber\\&\phantom{mm}- \frac12 \frac{(c\Delta t)^2}{\epsilon^2\Delta x \Delta y} \left( 
 - \frac{1}{2\pi} [[ [[u^n]]_{i\pm\frac12} ]]_{j\pm\frac12} - \frac14 [[v^n]_{i\pm1}]_{j\pm1}  + \frac14 [[[p^n]_{i\pm1}]]_{j\pm\frac12}
\right ) \label{eq:riemannscheme-u}\\
v^{n+1} &=v_{ij}^n - \frac{c\Delta t}{2\epsilon \Delta y}\left( [p^n]_{i,j\pm1} - [[ v^n]]_{i,j\pm\frac12} \right ) 
\nonumber\\&\phantom{mm}-\frac12 \frac{(c\Delta t)^2}{\epsilon^2\Delta x \Delta y } \left( 
- \frac1{2\pi} [[ [[v^n]]_{i\pm\frac12}]]_{j\pm\frac12}  -  \frac14 [[u^n]_{i\pm1}]_{j\pm1} + \frac14 [ [[p^n]]_{i\pm\frac12}]_{j\pm1} 
\right )\label{eq:riemannscheme-v}\\
p^{n+1} &= p^n_{ij} - \frac{c\Delta t}{2\epsilon \Delta x} \left( [u^n]_{i\pm1,j} - [[p^n]]_{i\pm\frac12,j} \right ) - \frac{c\Delta t}{2\epsilon \Delta y} \left( [v^n]_{i,j\pm1} - [[p^n]]_{i,j\pm\frac12} \right ) \nonumber\\
&\phantom{mm}- \frac12 \frac{(c\Delta t)^2}{\epsilon^2 \Delta x \Delta y } \left( 
\frac14 [[ [u^n]_{i\pm1}]]_{j\pm\frac12} +\frac14 [ [[v^n]]_{i\pm\frac12} ]_{j\pm1}-2 \cdot \frac1{2\pi}[[ [[p^n]]_{i\pm\frac12}]]_{j\pm\frac12}  
\right )\label{eq:riemannscheme-p}
\end{align}
This scheme is conservative because it is a Godunov scheme, and can be written as
\begin{align*}
q^{n+1} = q^n - \frac{\Delta t}{\Delta x} \left(f^{(x)}_{i+\frac12,j} - f^{(x)}_{i-\frac12,j}\right) -  \frac{\Delta t}{\Delta y} \left(f^{(y)}_{i,j+\frac12} - f^{(y)}_{i,j-\frac12} \right )
\end{align*}
One can identify the $x$-flux through the boundary located at $x_{i+\frac12}$:
\begin{align}
 f^{(x)}_{i+\frac12,{j}}&= \frac12 \frac{c}{\epsilon} \veccc{\{p\}_{i+\frac12, j} - [u]_{i+\frac12,j} }{0}{\{u\}_{i+\frac12,j} - [p]_{i+\frac12,j}}\nonumber \\&\phantom{mm}+ \frac12 \frac{c \Delta t }{\epsilon \Delta y} \cdot \frac{c}{\epsilon} \veccc{
   - \frac{1}{2\pi} [[ [u]_{i+\frac12}]]_{j\pm\frac12}  - \frac14   [ \{ v \}_{i+\frac12}  ]_{j\pm1}      + \frac14 [[\{p\}_{i+\frac12} ]]_{j\pm\frac12}  
 }
 {0 }
 { 
     \frac14  [[v]_{i+\frac12}]_{j\pm1}     - \frac{1}{2\pi}  [[[p]_{i+\frac12}]]_{j\pm\frac12}     } \label{eq:exactriemannflux}
\end{align}
The corresponding perpendicular flux is its symmetric analogue. The first bracket is the flux obtained in a dimensionally split situation, {i.e. when one-dimensional information is collected from different directions}. As the scheme is a Godunov scheme for piecewise constant reconstruction, it is first order in space and time.
 
The appearance of prefactors which contain $\pi$ in schemes derived using the exact multi-dimensional evolution operators has already been noticed in \cite{lukacova00}, but none of the schemes mentioned therein matches the one presented here.

For better comparison to other schemes, the scheme \eqref{eq:riemannscheme-u}--\eqref{eq:riemannscheme-p} can be rewritten in the variables prior to symmetrization, i.e. such that it is a numerical approximation to \eqref{eq:acousticscaledv}--\eqref{eq:acousticscaledp}. This is achieved by applying the transformation \eqref{eq:trafomatrix} or, which is equivalent, by replacing $p \mapsto \frac{p}{c \epsilon}$ everywhere.

Dimensionally split schemes in two spatial dimensions have a stability condition (\cite{godunov76}, Eq. 8.15, p. 63)
\begin{align*}
 c \Delta t < \frac{1}{\frac{1}{\Delta x} + \frac{1}{\Delta y}} 
\end{align*}
which for square grids gives a maximum \cfl number of 0.5. As the present scheme is an exact multidimensional Godunov scheme it is stable up to the physical \cfl number.

\subsection{Numerical examples} \label{ssec:numericalresults}

The scheme \eqref{eq:riemannscheme-u}--\eqref{eq:riemannscheme-p} is applied to several test cases. First, multi-dimensional Riemann Problems are considered, among them the one considered analytically in Section \ref{sec:riemann}. The last test is devoted to the low Mach number abilities of the scheme.

\vspace{0.1cm}
\subsubsection{Riemann Problems}

Two Riemann Problems are considered. The first setup is that of Section \ref{sec:riemann} (Fig. \ref{fig:initialdata}); it is solved on a square grid of $101 \times 101$ cells on a domain that is large enough such that the disturbance produced by the corner has not reached the boundaries for $t=0.25$. Here, $c = \epsilon = 1$. The results are shown in Fig. \ref{fig:riemannschemesolution}. In Fig. \ref{fig:vycomparison} the $y$-component of the velocity obtained with the numerical scheme is compared to the analytic solution \eqref{eq:vysolutionriemann} found in Section \ref{sec:riemann}. {The analytic solution is radially symmetric, this is why the numerical solution is plotted against the radius. The numerical error leads to a slight scatter of the points depending on the angle. For low \cfl\! numbers, the multi-dimensional scheme here does not seem to show any advantage. The stability domain of the dimensionally split scheme, {i.e. of the scheme which combines solutions of one-dimensional problems in different directions}, however, only extends up to $\cfl\!=0.5$. For high \cfl\! numbers, the multi-dimensional scheme is able to capture the features of the solution slightly better. Here the advantage of the multi-dimensional scheme becomes obvious -- the increased stability region allows to run computations with a $\cfl\!$ condition close to the physical limit.

A set of multi-dimensional Riemann Problems has been presented in \cite{amadori2015}. The Riemann Problem no. 3, p. 101 is given by
\begin{align*}
 p_0(\vec x ) &= 0 & u_0(\vec x) &= v_0(\vec x) = \sign(xy)
\end{align*}
This Riemann Problem is chosen as another test case. Fig. 6.6 in \cite{amadori2015} shows the analytic solution. Here, this setup is computed with the multi-dimensional Godunov scheme and the dimensionally split upwind scheme on a $51 \times 51$ grid. For the former, again, a high \cfl\! number can be chosen. Figure \ref{fig:riemanngosse} shows the results. One observes that the multi-dimensional scheme is able to resolve the features of the multi-dimensional interaction region much more sharply than the dimensionally split upwind scheme. 

Additionally, the dimensionally split scheme produces incorrect jumps in the central region, shown in Figure \ref{fig:riemanngossecenter}. The multi-dimensional scheme yields a solution of much better quality with no detectable incorrect jumps. In this case the multi-dimensional Godunov scheme is clearly superior because it takes into account truly multi-dimensional interactions directly, and not via one-dimensional steps.

\begin{figure}[h]
 \centering
\includegraphics[width=0.32\textwidth]{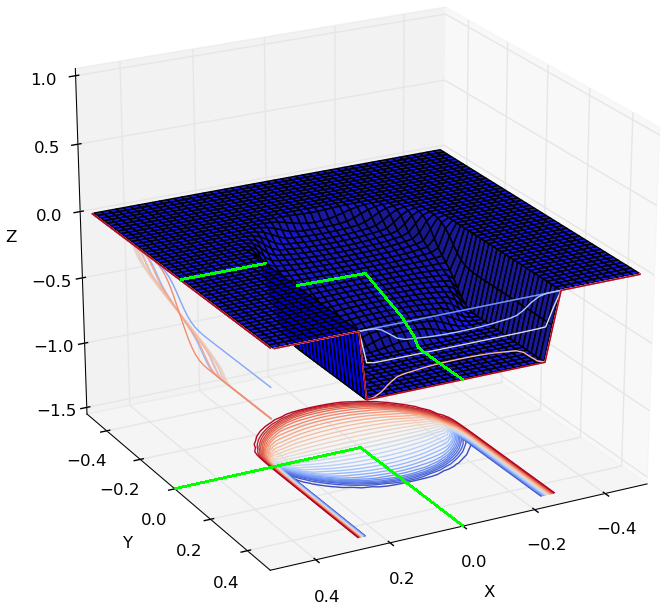}
\includegraphics[width=0.32\textwidth]{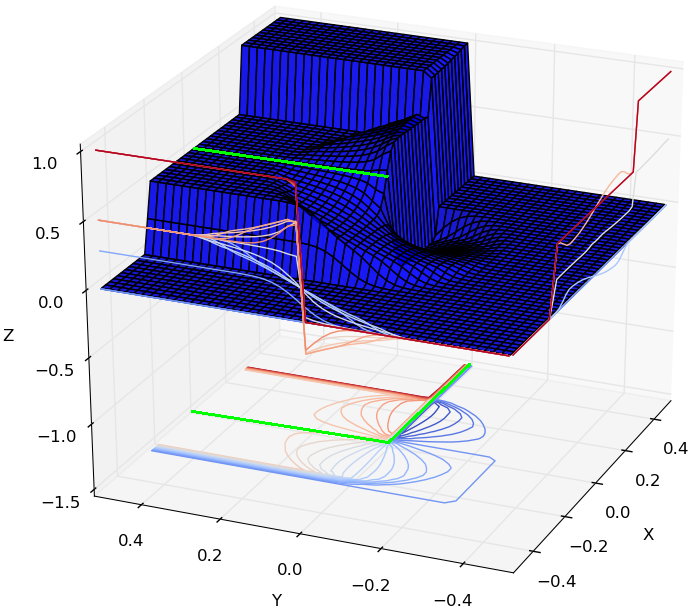}
\includegraphics[width=0.32\textwidth]{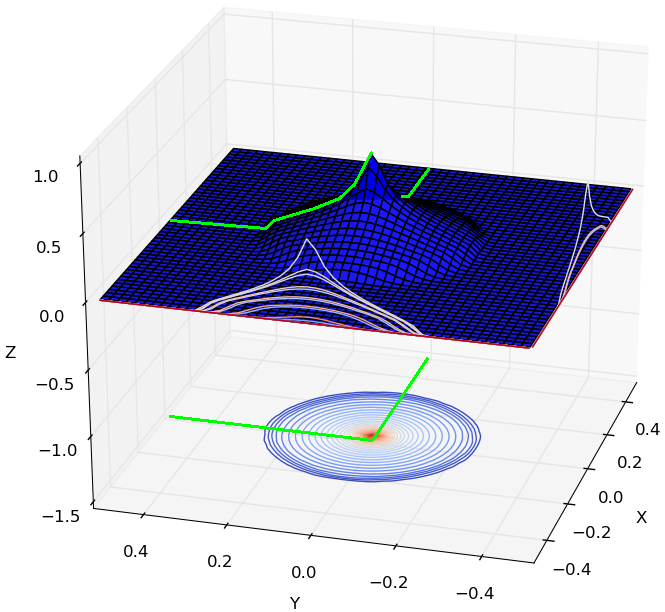}
\caption{Results of a numerical computation of the Riemann problem of Section \ref{sec:riemann} at time $ct = 0.25$ using scheme \eqref{eq:riemannscheme-u}--\eqref{eq:riemannscheme-p}. {\sl Left}: Pressure. {\sl Center}: $x$-velocity. {\sl Right}: $y$-velocity. Compare the images to Fig. \ref{fig:velxNE}. The $\cfl\!$ number has been chosen very close to 1; for small values the discontinuities are smoothed out more.}
\label{fig:riemannschemesolution}
\end{figure}

\begin{figure}[h]
 \centering
\includegraphics[width=0.5\textwidth]{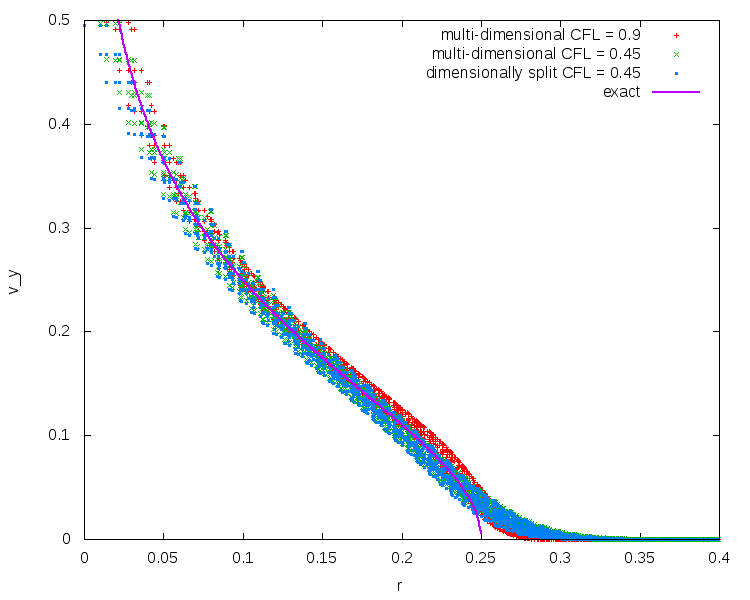}\phantom{mmmmmmmmmmmmm}
\caption{{The $y$-component of the velocity obtained by the numerical scheme \eqref{eq:riemannscheme-u}--\eqref{eq:riemannscheme-p} for \cfl\! numbers 0.9 and 0.45 is shown together with the result obtained by a dimensionally split scheme ($\cfl\!=0.45$) and the analytic solution \eqref{eq:vysolutionriemann}. The values are plotted against the radius $\sqrt{x^2 + y^2}$. One observes that the multi-dimensional scheme slightly outperforms the dimensionally split one, because it can be run with high \cfl\! numbers.}}
\label{fig:vycomparison}
\begin{picture}(0,0)
 \put(10,105){\includegraphics[width=0.4\textwidth]{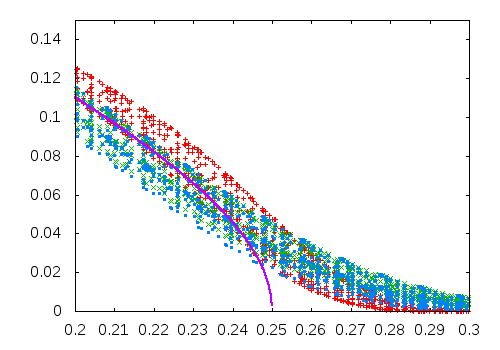}}
\end{picture}
\end{figure}

\begin{figure}[h]
 \centering
\includegraphics[width=0.25\textwidth]{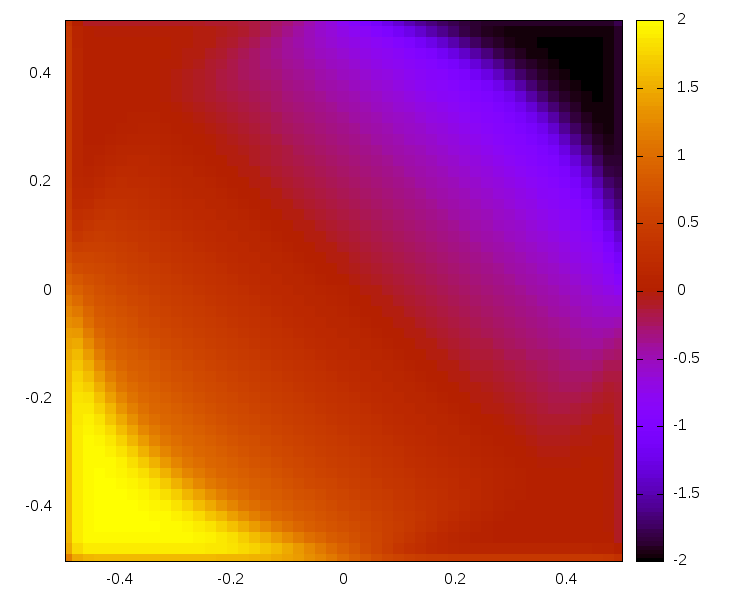}
\includegraphics[width=0.25\textwidth]{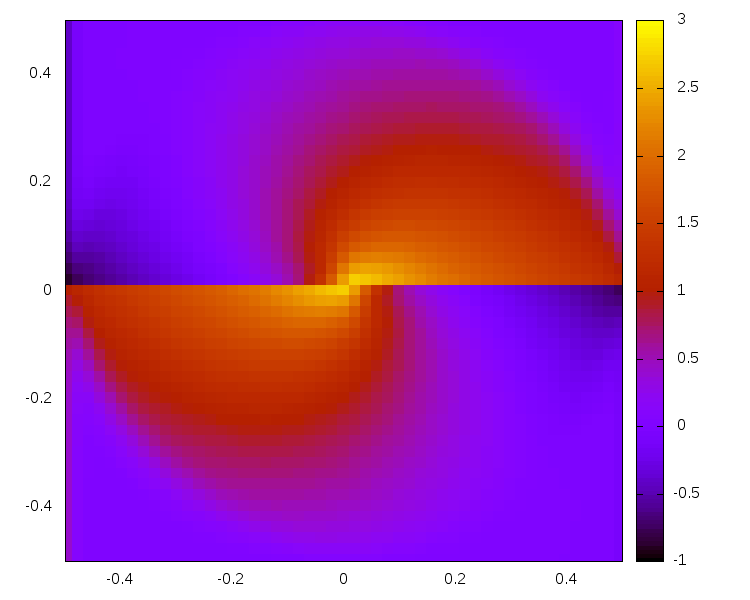}
\includegraphics[width=0.25\textwidth]{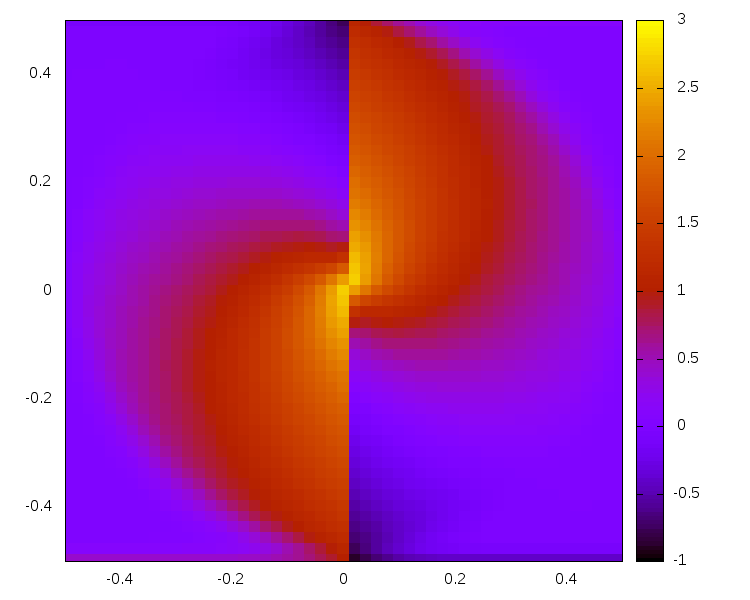}\\
\includegraphics[width=0.25\textwidth]{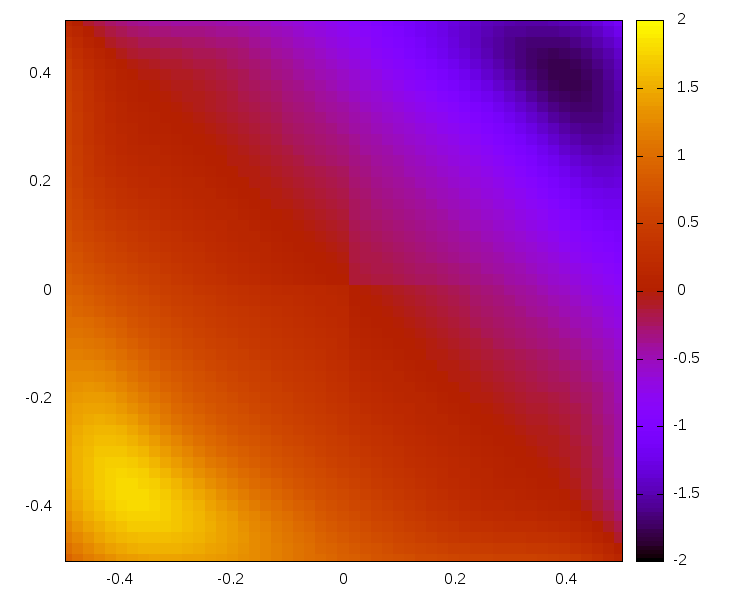}
\includegraphics[width=0.25\textwidth]{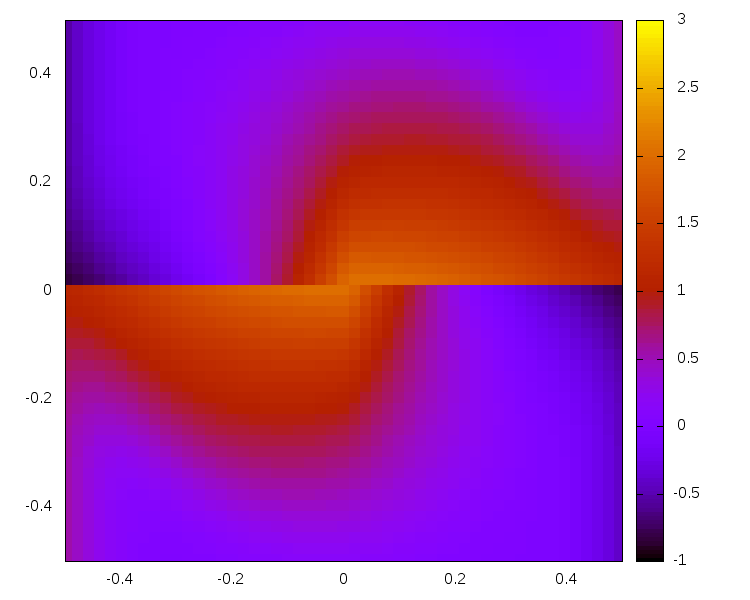}
\includegraphics[width=0.25\textwidth]{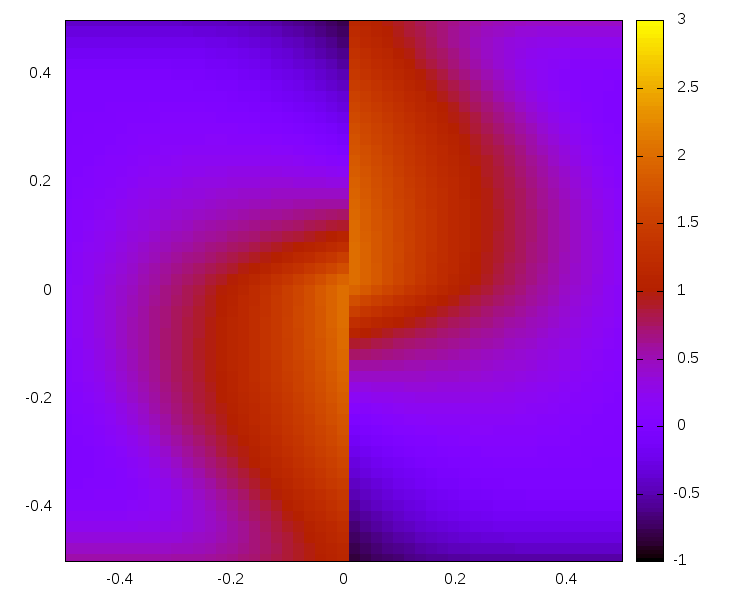}
\caption{{Results of a numerical computation of the Riemann problem from \cite{amadori2015}. The numerical solution shown at time $ct = 0.5$. {\sl Top row}: Multi-dimensional scheme \eqref{eq:riemannscheme-u}--\eqref{eq:riemannscheme-p} with $\cfl\!=0.98$, {\sl Bottom row}: Dimensionally split scheme with $\cfl\!=0.45$. {\sl Left}: Pressure. {\sl Center}: $x$-velocity. {\sl Right}: $y$-velocity.}}
\label{fig:riemanngosse}
\end{figure}

\begin{figure}[h]
 \centering
\includegraphics[width=0.32\textwidth]{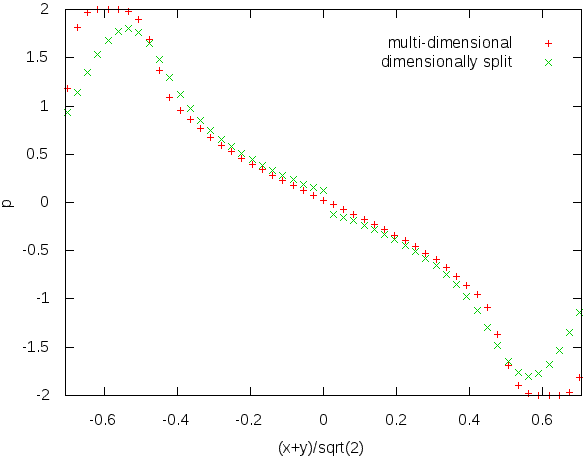}\hfill
\includegraphics[width=0.32\textwidth]{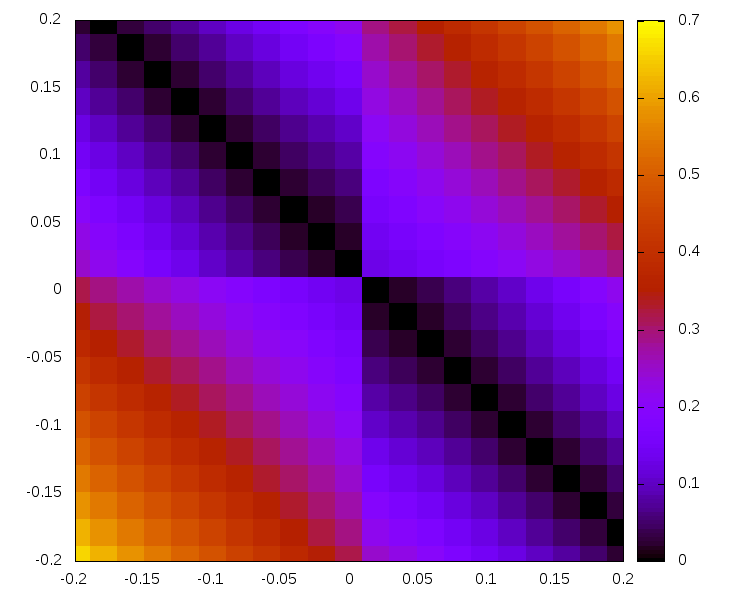}\hfill
\includegraphics[width=0.32\textwidth]{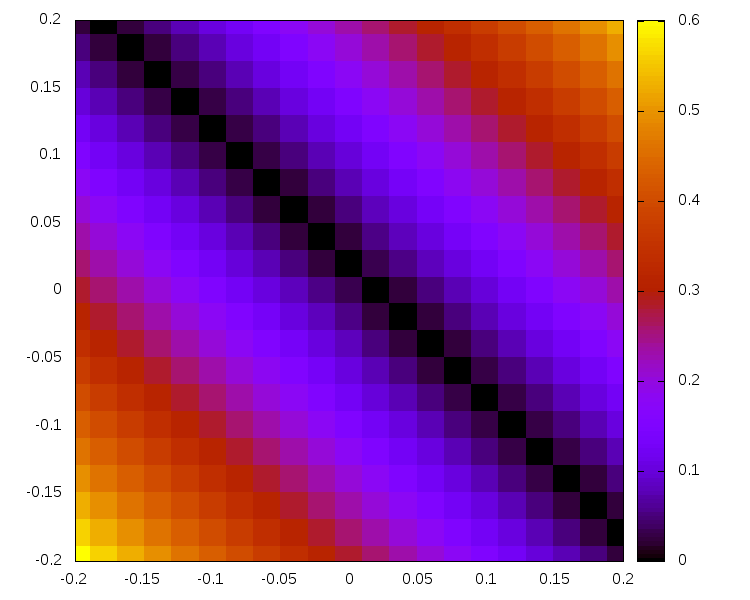}
\caption{{{\sl Left}: The pressure at $ct = 0.5$ for the Riemann problem from \cite{amadori2015} shown along the diagonal $x = y$ of the grid. The result obtained with the multi-dimensional scheme \eqref{eq:riemannscheme-u}--\eqref{eq:riemannscheme-p} with $\cfl\!=0.98$ is compared to that of the dimensionally split scheme with $\cfl\!=0.45$. For the latter one observes an unphysical jump and excessive smoothing. {\sl Center:} The jumps are visible in the two dimensional plot as well. Color coded is the absolute value $|p|$ for better visibility of the artefact. {\sl Right}: The same region computed with the multi-dimensional Godunov scheme.}}
\label{fig:riemanngossecenter}
\end{figure}

}

\subsubsection{Low Mach number vortex}

The second test shows the properties of the scheme in the limit $\epsilon \to 0$. The setup is that of a stationary, divergencefree velocity field and constant pressure:
\begin{align*}
 p_0(\vec x) &= 1\\
 \vec v_0(\vec x) &= \vec e_\phi \begin{cases} \frac{r}{r_0} & r < r_0 \\ 2-\frac{r}{r_0} & r_0 \leq r < 2 r_0 \\ 0 & \text{else} \end{cases}
\end{align*}
The velocity thus has a compact support, which is entirely contained in the computational domain, discretized by $51 \times 51$ square cells. Here $c=1$ and $r_0=0.2$. Zero-gradient boundaries are enforced. Results of a simulation are shown in Figure \ref{fig:riemannschemevortex}.

\begin{figure}[h]
 \centering
 \includegraphics[width=0.32\textwidth]{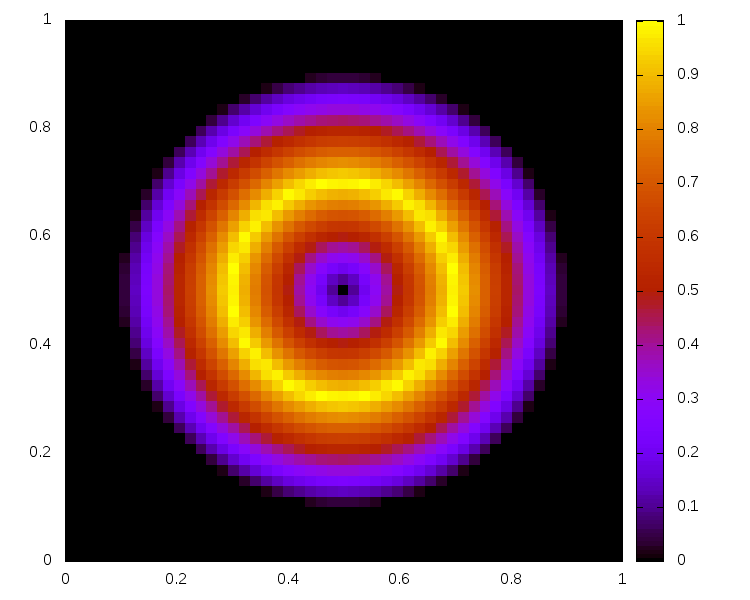}
\includegraphics[width=0.32\textwidth]{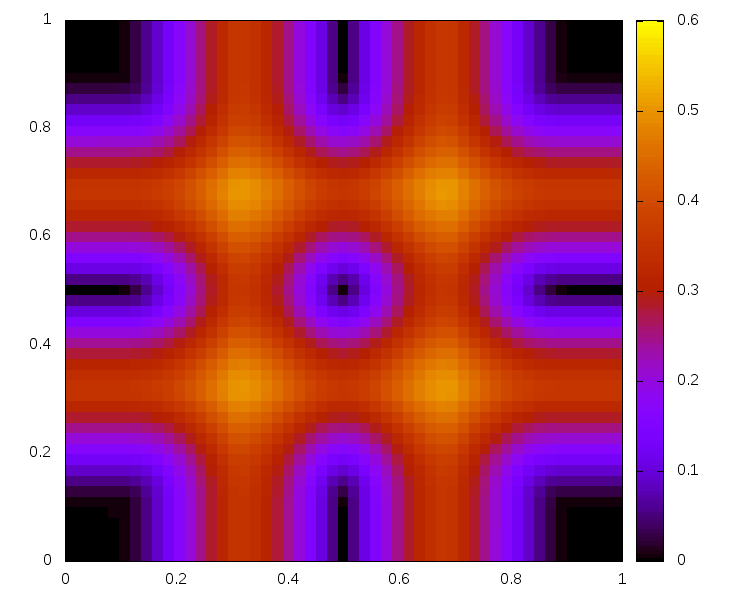}%\hfill
\includegraphics[width=0.32\textwidth]{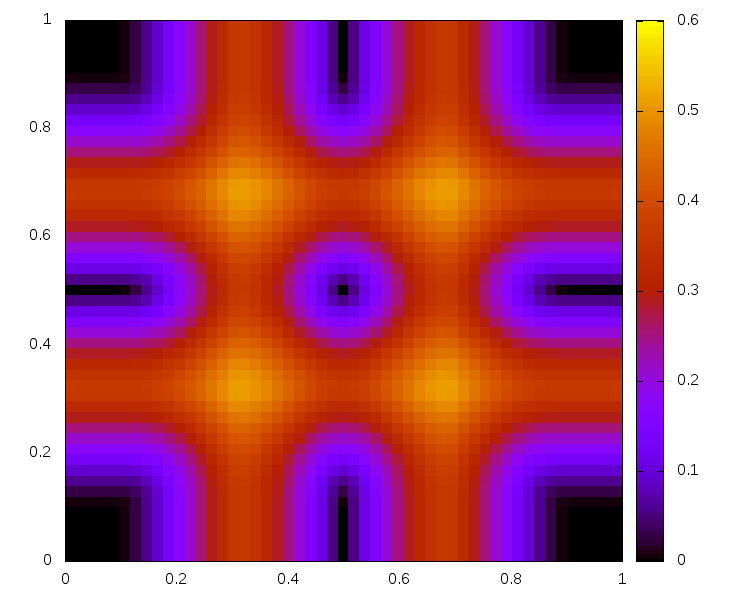}
\caption{Solution of the vortex initial data at time $ct = 1$ for $\epsilon = 10^{-2}$. The quantity shown in colour is the magnitude $|\vec v|$ of the velocity. {\sl Left}: Exact solution $=$ initial data. {\sl Center}: Multi-dimensional Godunov scheme \eqref{eq:riemannscheme-u}--\eqref{eq:riemannscheme-p} with \cfl = 0.8. {\sl Right}: Dimensionally split scheme with \cfl = 0.45.}
\label{fig:riemannschemevortex}
\end{figure}

The dimensionally split solver is known to display artefacts in the limit $\epsilon \to 0$ (see e.g. \cite{guillard99,barsukow17a}). {For comparison, results obtained with the dimensionally split scheme are shown in the same Figure \ref{fig:riemannschemevortex}. However, as the dimensionally split scheme is only stable up to $\cfl\! = 0.5$, it is not possible to run the setup with the same $\cfl\!$ number. 

Knowing that the dimensionally split scheme fails to resolve the low Mach number limit, a visual comparison of the two results indicates that the multi-dimensional Godunov scheme is equally unable to resolve it. This can be theoretically confirmed. In \cite{barsukow17a}, a strategy has been presented how low Mach compliance can be checked theoretically for linear schemes for acoustics. Therein the concept of \emph{stationarity preservation} has been introduced. A scheme is called stationarity preserving if it discretizes the entire set of the analytic stationary states. It is found that most schemes add so much diffusion, that only trivial (e.g. spatially constant) stationary states are not diffused away. It is moreover found that the low Mach number limit for linear acoustics is equivalent to the limit of long times. In order to study the limit $\epsilon \to 0$ one therefore needs to analyze the long time behaviour of the numerical solutions. Here, the most prominent role is played by the stationary states. Only if they are discretized correctly will the scheme be low Mach compliant. In \cite{barsukow17a} a condition has been found, which involves the so-called \emph{evolution matrix}: The scheme is stationarity preserving if the determinant of its evolution matrix vanishes. Given a numerical scheme, its evolution matrix can be easily constructed. For more details the reader is referred to \cite{barsukow17a} or \cite{barsukow17lilleproceeding}. For the multi-dimensional Godunov scheme under consideration it can be verified by explicit calculation that the evolution matrix fails to meet the condition of stationarity preservation. The computation is lengthy, and is thus not reproduced here. The inability of the multi-dimensional Godunov scheme to resolve the limit of low Mach number, however, thus can also be understood theoretically.

The multi-dimensional Godunov scheme is not low Mach compliant. However, no approximations were made in the evolution step of the scheme. It thus becomes obvious that the low Mach number problem is not cured by just taking into account all the multi-dimensional interactions. As the evolution step was exact, the reason for the failure is to be sought elsewhere. {The only approximation used in this scheme is the choice of a piecewise constant reconstruction.}. {A number of ``fixes'' have been suggested for the low Mach number regime of the dimensionally split Roe scheme (\cite{thornber08,dellacherie10,chalons13,li13,birken16,barsukow16}). They might in principle be used here as well, but this would imply going back to an approximate evolution operator, and they shall not be considered here.} 

{The result of this section implies that for deriving low Mach number numerical schemes either an approach different from the Godunov procedure needs to be found, or that no first-principles derivation is available, but only ``fixes''.}
}

\section{Conclusions and outlook}
\label{sec:outlook}

This paper presents an exact solution of the acoustic equations in three spatial dimensions. It is a paramount example of a solution operator that is very different to that of multi-dimensional linear advection: characteristic cones and spherical means replace simple transport along a one-dimensional characteristic. On the other hand it also shows differences to the well-understood scalar wave equation, thus emphasizing the additional difficulties when dealing with systems of equations. The solution obtained is a distributional one, which allows to use discontinuous initial data. 

The multi-dimensional Riemann Problem for the acoustic equations is an example for the appearance of a singularity in the solution, which is an intrinsically multi-dimensional feature. This paper presents the exact shape of the solution and shows that the singularity is logarithmic.

Furthermore, using the exact solution formulae, a multi-dimensional Godunov scheme is obtained. It has been found in experiments that, as expected, its stability region extends right up to the maximum allowed physical \cfl number, which is twice what is possible with a dimensionally split scheme. 
The method has been shown to perform well when applied to multi-dimensional Riemann Problems, not suffering from the artefacts found for dimensionally split schemes. It, however, does not allow calculations in the limit of low Mach numbers in general. It fails to resolve the limit even though the evolution step of the Godunov scheme is exact. 

All the necessary results for the derivation of the numerical scheme could be obtained analytically. This shows that the exact solution operator can be efficiently used in such circumstances. The exact evolution operator may be an important ingredient in order to endow the numerical scheme with certain, desirable properties. Future work will try to generalize aspects of these findings to the full Euler equations. We hope that the study of linear acoustics in multiple spatial dimensions is a first step in this direction.

\section*{Acknowledgement}

The authors thank Philip L. Roe and Praveen Chandrashekar for inspiring discussions and the anonymous Referee whose comments have helped improving the manuscript. WB acknowledges support from the German National Academic Foundation and thanks Marlies Pirner for valuable comments.

\newcommand{\etalchar}[1]{$^{#1}$}

\appendix

\section{Distributions} \label{ssec:distributions}

In order to use discontinuous initial data for the equations of linear acoustics in multiple spatial dimensions one needs distributional solutions. This is not necessary in one spatial dimension. In multiple spatial dimensions, however, the solution formula involves derivatives of the initial data. If the data are discontinuous, the solution formula thus needs to be generalized in the sense of distributions.

In this Section a brief review of definitions and results from the theory of distributions is given. This is done in order to present the notation that will be used. Therefore many standard results are stated without proofs. The reader interested in a thorough introduction is, for example, referred to \cite{schwartz78,gelfand64,hormander13,rudin91}. 

In Section \ref{sec:distributionderivation} distributional solutions of linear acoustics in three spatial dimensions are derived.

\begin{definition}[Distribution] \label{def:distribution}
 A \emph{distribution} is a continuous linear functional on the set $D(\mathbb R^d)$ of compactly supported smooth test functions $\test$. The evaluation of the distribution $f$ on a test function $\test$ is denoted by $\langle f | \test\rangle \in \mathbb R$ (or $\mathbb C$). The set of all distributions is denoted by $D'(\mathbb R^d)$.
\end{definition}

It is possible to show (see e.g. \cite{rudin91}) that a function $h \in L^1_\text{loc}(\mathbb R^d)$ gives rise to a distribution $\regdist{h} \in D'(\mathbb R^d)$ in the following way: the action $\langle \regdist{h} | \test \rangle$ of $\regdist{h}$ onto any test function $\test \in D(\mathbb R^d)$ is defined as:
 \begin{align}
  \langle \regdist{h} | \test \rangle := \int_{\mathbb R^d} \dd \vec x\, h(\vec x)\test(\vec x)  \label{eq:regdist}
 \end{align}

\begin{definition}[Regular distribution]
 Given $h \in L^1_\text{loc}(\mathbb R^d)$, the distribution $\regdist{h}$, defined by its action onto a test function $\test \in D(\mathbb R^d)$ as in \eqref{eq:regdist}, is called \emph{regular distribution}.
\end{definition}

In order to make explicit the independent variable, the notation $\langle \regdist{h} | \test(\vec x) \rangle$ will be used. Two regular distributions $\regdist{h_1}$ and $\regdist{h_2}$ are equal, if $h_1 = h_2$ almost everywhere.

\begin{definition}[Tempered distribution]
 The \emph{Schwartz space} $S(\mathbb R^d)$ of rapidly decreasing functions $f$ on $\mathbb R^d$ is defined as 
 \begin{align*}
  S(\mathbb R^d) := \Big \{ f \in C^\infty(\mathbb R^d) : \sup_{\vec x \in \mathbb R^d} |x_1^{a_1} \dots x_d^{a_2} \del_{x_1}^{b_1} \cdots \del_{x_d}^{b_d} f | < \infty \nonumber\\ \forall (a_1, \ldots, a_d, b_1, \ldots, b_d) \in \mathbb (\mathbb N_0)^{2d} \Big \}
 \end{align*}
 The set $S'(\mathbb R^d)$ of \emph{tempered distributions} is the continuous dual of $S(\mathbb R^d)$.
\end{definition}

It is possible to show that the derivative $\nabla_{\vec x} T$ of a distribution $T \in D'(\mathbb R^d)$, defined in the following, is again a distribution (see e.g. \cite{rudin91}).

\begin{definition} \label{def:fourier}
 \begin{enumerate}[i)]
  \item The derivative $\nabla_{\vec x} T$ of a distribution $T \in D'(\mathbb R^d)$ is defined as $$\langle \nabla_{\vec x} T | \test(\vec x) \rangle := - \langle T | \nabla_{\vec x} \test(\vec x) \rangle \qquad \forall \test \in D(\mathbb R^d)$$
  \item The Fourier transform {$\mathbb F_{\vec x}$} applied to an integrable function {$f : \mathbb R^d \to \mathbb R$} is defined by
  \begin{align*}
   \hat f(\vec k) := \mathbb F_{\vec x}[f](\vec k) :=
   \frac{1}{(2 \pi)^{d/2}} \int \dd \vec x \exp(
   %- \ii \omega t + 
   \ii \vec k \cdot \vec x) f(\vec x)
  \end{align*}
   Generically, $\vec k$ denotes the dual variable to $\vec x$.
   Note the symmetric prefactor convention chosen here.
   Also, generically, the hat denotes a Fourier transform in the following.
  \item The Fourier transform {$\mathbb F_{\vec x}[T](\vec k)$} of a distribution {$T \in S'(\mathbb R^d)$} is defined by
  \begin{align*}
  \langle \mathbb F[T]| \test \rangle :=  \langle T| \mathbb F[\test] \rangle\qquad \forall \test \in {S(\mathbb R^{d})}
  \end{align*}
  or, making explicit the independent variables,
  \begin{align*} \langle \mathbb F_{\vec x}[T](\vec k) | \test(\vec k) \rangle :=  \langle T(\vec x) | \mathbb F_{\vec k}[\test](\vec x) \rangle\qquad \forall \test \in S(\mathbb R^{d}) \end{align*}
 \end{enumerate}
\end{definition}

The class ${S(\mathbb R^{d})}$ allows to put the Fourier transform to maximal use:
\begin{theorem} \label{thm:fourierautom}
 The Fourier transform is an automorphism on ${S'(\mathbb R^{d})}$.
\end{theorem}
\noindent
For a proof see e.g. \cite{rudin91}. The usual rules of differentiation apply:

\begin{theorem} 
Consider a distribution $q \in S'(\mathbb R^d)$ and its Fourier transform $\hat q$. Then 
\begin{enumerate}[i)]
  \item \begin{align} \nabla_{\vec x} \mathbb F^{-1}_{\vec k}[\hat q(\vec k)] =  F^{-1}_{\vec k}[\ii \vec k \hat q(\vec k)] \label{example-diffdistr}\end{align}
  \item \begin{align}
  \mathbb F_{\vec k}^{-1}[\nabla_{\vec k} \hat q]  = -\ii \vec x \mathbb F_{\vec k}^{-1}[\hat q] \label{eq:fourier-internal-derivative} \\
  \mathbb F_{\vec x}[\nabla_{\vec x} q] = \ii \vec k \mathbb F_{\vec x}[q] \label{eq:fourierderiv}
  \end{align}
\end{enumerate}
\end{theorem}
\begin{proof}
\begin{enumerate}[i)]
\item For every test function $\test \in S(\mathbb R^d)$
\begin{align*}
 \langle \nabla_{\vec x} \mathbb F^{-1}_{\vec k}[\hat q(\vec k)] | \test(\vec x) \rangle 
 = -\langle \hat q(\vec k) | \mathbb F^{-1}_{\vec x}[\nabla_{\vec x} \test(\vec x)] \rangle  
 = \langle  \mathbb F^{-1}_{\vec k}[\ii \vec k \hat q(\vec k)] |  \test \rangle 
\end{align*}
\item Analogously, 
\begin{align*}
 \langle \mathbb F_{\vec k}^{-1}[\nabla_{\vec k} \hat q(\vec k)] | \test(\vec x) \rangle &= \langle \nabla_{\vec k} \hat q(\vec k) | \mathbb F_{\vec x}^{-1}[\test] \rangle =- \langle \hat q(\vec k) | \nabla_{\vec k}  \mathbb F_{\vec x}^{-1}[\test] \rangle \\
 &=- \langle \hat q(\vec k) | \mathbb F_{\vec x}^{-1}[\ii \vec x \test] \rangle = \langle -\ii \vec x\mathbb F_{\vec k}^{-1}[\hat q(\vec k)] |  \test(\vec x) \rangle 
\end{align*}
The other equation is shown by repeating the argumentation for $\mathbb F_{\vec x}[q]$.
\end{enumerate}
\end{proof}

The Fourier transform of 1 is (up to factors) the Dirac distribution $\delta_0$:
\begin{definition}[Dirac distribution]
The Dirac distribution $\delta_{\vec x'}$, or $\delta_{\vec x=\vec x'}$, is defined as $\langle\delta_{\vec x'}|\test\rangle := \test(\vec x')$ $\forall \test \in D(\mathbb R^d)$.
\end{definition}

For distributional (generalized) solutions to partial differential equations see e.g. \cite{john78,rauch91,taylor96,zuily02}.

\begin{definition}[Distributional solution]
 Consider a first order linear differential operator $\mathcal L$ {with constant coefficients} containing derivatives with respect to $t \in {\mathbb R^+_0}$ and $\vec x \in \mathbb R^d$. Then $q \in C^1({\mathbb R^+_0},D'(\mathbb R^d))$ is called a \emph{distributional solution} of the initial value problem $\mathcal L q = 0$ {with initial data $q_0 \in D'(\mathbb R^{d})$} if $\mathcal L q = 0$ holds for every $t$ as an identity in {$D'(\mathbb R^{d})$}; i.e. if $\langle (\mathcal L q)(t) | \test \rangle = 0$ $\forall \test \in {D(\mathbb R^{d})}$ $\forall t$ and $q(0) = q_0$. 
\end{definition}

This definition extends analogously to systems of PDEs.

In general in the following the solution will not be a function, but the initial data $q_0$ will. If the initial data are locally integrable functions, then in the context of a \emph{distributional} initial value problem they are to be interpreted as regular distributions $\regdist{q_0}$.

The convolution $F * G$ of two distributions $F$ and $G$ can be defined in certain cases. Here only the following definition is needed, and the reader is referred to e.g. \cite{rudin91} for further details.

\begin{definition}[Convolution] \label{def:convolution}
 The convolution $F * G$ of $F,G \in D'(\mathbb R^d)$, with at least one of them having compact support, is defined $\forall \test \in D(\mathbb R^d)$ as
 
 \begin{align*}
  \langle (F * G)(\vec x) | \test(\vec x) \rangle = \Big\langle F(\vec x) | \langle G(\vec y) | \test(\vec x + \vec y) \rangle \Big\rangle
 \end{align*}
\end{definition}

If $F$ and $G$ are regular distributions, i.e. $F = \regdist{f}$, $G = \regdist{g}$, with $f, g$ having compact support, then
\begin{align*}
 \langle \regdist{f} * \regdist{g} | \test \rangle &= 
 \int \dd x f(x) \int \dd y g(y) \test(x+y) 
 = \int \dd \xi \left( \int \dd y f(\xi - y) g(y) \right )\test(\xi)
\end{align*}

It can be shown that $\delta_0$ acts as the identity upon convolution, and $\delta_{\vec x'}$ as translation by $\vec x'$. For $F, G \in S'(\mathbb R^d)$ and at least one of them compactly supported, the product of Fourier transforms is the Fourier transform of the convolution:
\begin{align*}
 \mathbb F_{\vec x}[F](\vec k) \cdot \mathbb F_{\vec x}[G](\vec k) &= \frac{1}{(2\pi)^{d/2}} \mathbb F_{\vec x}[F * G](\vec k)
\end{align*}

{As will be seen below, the spherical symmetry of the exact solution for the acoustic equations manifests itself in the appearance of $|\vec k|$ in the Fourier transforms. This gives rise to some very special distributions which are discussed next.

\begin{definition}[Radial Dirac distribution and step function] \label{def:radialdelta}
 Choose $r \in \mathbb R^+$ and $d \in \mathbb N^+$.
 \begin{enumerate}[i)]
 \item The \emph{radial Dirac distribution} $\delta_{|\vec x|=r}$ is defined as
\begin{align*}
 \Big\langle \delta_{|\vec x|=r} |\, \test(\vec x)\Big\rangle := \int_{S_r^{d-1}} \!\! \dd \vec x \, \test(\vec x) \qquad \forall \test \in D(\mathbb R^d)
\end{align*}
 \item Define the characteristic function $\Theta_{|\vec x|\leq r}$ of the ball $B_r^{d}$
 { \begin{align*}
  \Theta_{|\vec x|\leq r} := \begin{cases} 1 & \vec x \in B_r^d \\ 0 & \text{else}  \end{cases}
 \end{align*} }
\end{enumerate}
\end{definition}

\begin{definition}[Spherical average]\label{def:sphericalaverage}
In three spatial dimensions, the spherical average at a radius $r$ of a distribution $T$ is given by $$\frac{1}{4\pi} \frac{\delta_{|\vec x|=r}}{r^2} * T$$
\end{definition}
If $T$ is a regular distribution $T = \regdist{f}$, then by Definition \ref{def:convolution} $\forall \test \in S(\mathbb R^3)$
\begin{align*}
 \frac{1}{4\pi} \left \langle \frac{\delta_{|\vec x|=r}}{r^2} * T \Big | \test \right \rangle &= \frac{1}{4\pi} \frac{1}{r^2} \int_{S^2_r} \dd \vec x \int \dd \vec y f(\vec y) \test(\vec x + \vec y)\\
 &= \left \langle \regdist{  \frac{1}{4 \pi}  \int_{S^2_1} \dd \vec y f(\vec x + r \vec y) }\Big |  \test(\vec x) \right \rangle
 \end{align*}
Here, $\int_{S^2_1} \dd \vec y $ denotes an integration over the surface of a 2-sphere of radius 1, i.e. in spherical polar coordinates this amounts to
\begin{align*}
 \int_0^\pi \dd \theta \sin \theta \int_0^{2\pi} \dd \phi
\end{align*}

{For more details on properties of spherical averages, see \cite{john81}.}

\begin{theorem}[Radial Dirac distribution] \label{thm:deriavtiveradialstep}
 The derivative of the radial step function $\Theta_{|\vec x| \leq r}$ is closely related to the radial Dirac distribution:
 \begin{align*}
  -\delta_{|\vec x|=r} \frac{\vec x}{|\vec x|} &= \nabla \regdist{\Theta_{|\vec x| \leq r}}\\
  \intertext{which can, by defining $\del_r := \frac{\vec x}{|\vec x|} \cdot \nabla$ be rewritten as}
  -\delta_{|\vec x|=r} &= \del_r \regdist{\Theta_{|\vec x| \leq r} }
 \end{align*}
\end{theorem}
\begin{proof}
 Recall the definition of a ball $B^{d+1}_r = \{ \vec x \in \mathbb R^{d+1} : |\vec x| \leq r \}$ and abbreviate $\vec n := \frac{\vec x}{|\vec x|}$. Use Definition \ref{def:radialdelta} and Gauss' Theorem, for any $\test \in D(\mathbb R^d)$:
 \begin{align*}
  -\langle \delta_{|\vec x|=r} \vec n | \test \rangle &= - \int_{S^d_r} {\dd \vec y \frac{\vec y}{|\vec y|}} \cdot \test  = - \int_{B^{d+1}_r} \dd \vec x \, \nabla \cdot \test = - \langle \regdist{\Theta_{|\vec x|\leq r}} | \nabla \test \rangle \\&= \langle \nabla \regdist{\Theta_{|\vec x| \leq r}} | \test \rangle
 \end{align*}
 Multiplying through with $\vec n$ proves the assertion.
\end{proof}

\begin{theorem}[Fourier transforms] \label{thm:fourierradial}
 \begin{enumerate}[i)]
 \item In three spatial dimensions, given $r \in \mathbb R^+$, $\vec k \in \mathbb R^3$, 
 \begin{align*}
  \int_{S_r^2} \dd \vec x \,\exp(\ii \vec k \cdot \vec x) = 4 \pi r^2 \frac{\sin(|\vec k|r)}{|\vec k|r}
 \end{align*} \label{it:lemma1}
 \item In three spatial dimensions, the Fourier transform of the radial Dirac distribution $\delta_{|\vec x|=r}$ 
 is given by
 \begin{align*}
  \mathbb F_{\vec x}[\delta_{|\vec x|=r}](\vec k) = \frac{2}{(2 \pi)^{1/2}}    r^2 \frac{\sin (|\vec k|r)}{|\vec k|r}
 \end{align*}
 \item In three spatial dimensions, the Fourier transform of $\Theta_{|\vec x|\leq r} \frac{1}{|\vec x|}$ is given by
 \begin{align*}
  \mathbb F_{\vec x}\left[ \frac{\Theta_{|\vec x|\leq r}}{|\vec x|}\right](\vec k) = -\frac{2}{(2 \pi)^{1/2} } \frac{\cos (|\vec k| r) - 1}{|\vec k|^2}
 \end{align*}
 \end{enumerate}
\end{theorem}
\begin{proof}
 \begin{enumerate}[i)]
  \item Integrating in spherical polar coordinates:
 \begin{align*}
  \int_{S^2_r} \dd \vec x \,\exp(\ii \vec k \cdot \vec x) &= r^2 \int_0^\pi \dd \theta \,\sin\theta \int_0^{2\pi} \dd \phi \, \exp(\ii |\vec k| r \cos \theta)
  = 4 \pi r^2 \frac{\sin( |\vec k| r )}{|\vec k| r}
 \end{align*}
 \item Using \ref{it:lemma1}) and Definitions \ref{def:radialdelta} and \ref{def:fourier}, for any $\test \in S(\mathbb R^3)$
 \begin{align*}
	\langle \mathbb F_{\vec x}[\delta_{|\vec x|=r}] | \test \rangle &= \langle \delta_{|\vec x|=r} | \mathbb F_{\vec k}[\test] \rangle
	= \int_{S_r^2} \dd \vec x \frac{1}{(2 \pi)^{3/2}} \dd \vec k \exp(-\ii \vec k \cdot \vec x) \test(\vec k) \\
	&= \Big\langle \frac{1}{(2 \pi)^{3/2}} \int_{S_r^2} \dd \vec x \exp(-\ii \vec k \cdot \vec x) \Big| \test \Big\rangle \\
	&= \Big\langle \frac{2}{(2 \pi)^{1/2}}    r^2 \frac{\sin (|\vec k| r)}{|\vec k| r} \Big| \test \Big\rangle
       \end{align*}
\item Note that $\Theta_{|\vec x|\leq r} \frac{1}{|\vec x|}$ is an $L^1_\text{loc}$ compactly supported function in three spatial dimensions. Thus, 
 \begin{align*}
        \mathbb F_{\vec x}\left[\frac{\Theta_{|\vec x|\leq \rho}}{|\vec x|} \right](\vec k) &= \frac{1}{(2 \pi)^{3/2}} \int_0^\rho \dd r \frac{1}{r} \int_{|\vec x| = r} \dd \vec x \exp(-\ii \vec k \cdot \vec x) \\&=  \frac{2}{(2 \pi)^{1/2} |\vec k|} \int_0^\rho \dd r \sin (|\vec k| r)
        = -\frac{2}{(2 \pi)^{1/2} } \frac{\cos (|\vec k| \rho) - 1}{|\vec k|^2}
       \end{align*}
 \end{enumerate}
\end{proof}

}

\section{Derivation of the solution formulae} \label{sec:distributionderivation}

The very standard procedure of finding a solution to any linear equation such as \eqref{eq:system} for sufficiently regular initial data $q_0(\vec x)$ is to decompose them into Fourier modes
\begin{align*}
 q_0(\vec x) &= \frac{1}{(2\pi)^{d/2}} \int \dd \vec k\, {\hat q}_0(\vec k) \exp(\ii \vec k \cdot \vec x)
\end{align*}
where $d \in \mathbb N$ is the dimensionality of the space. The coefficients ${\hat q}_0(\vec k) = (\hat {\vec v}_0(\vec k), \hat p_0(\vec k))$ of this decomposition are the Fourier transform of $q_0$ and $\vec k$ here characterizes the mode. The time evolution of any single Fourier mode can be found via the ansatz
\begin{align*}
 T_t\Big ( \exp(\ii \vec k \cdot \vec x)  \Big )  =  \exp\Big(- \ii \omega(\vec k) t + \ii \vec k \cdot \vec x\Big)
\end{align*}
where the function $\omega(\vec k)$ is to be determined from the equations by inserting the ansatz. {For \eqref{eq:acousticv}--\eqref{eq:acousticp} one finds $\omega(\vec k) \in \{ 0, \pm c |\vec k|\}$.} The time evolution $q(t, \vec x)$ of the initial data $q_0(\vec x)$ is given by {summing} all the time evolutions of the individual modes.
For {the IVP of} the acoustic system \eqref{eq:acousticv}--\eqref{eq:acousticp} {with initial data $(\vec v_0(\vec x), p_0(\vec x))$} the solution is
\begin{align}
 p(t, \vec x) = \frac{1}{(2\pi)^{d/2}} \int \dd \vec  k &\left(\frac{\hat p_0(\vec k) + \frac{\vec k \cdot  \hat {\vec v}_0(\vec k)}{|\vec k|} }{2}  \exp(\ii \vec k \cdot \vec x - \ii c |\vec k| t) \right. \nonumber\\ &+ \left . \frac{\hat p_0(\vec k) - \frac{\vec k \cdot  \hat {\vec v}_0(\vec k)}{|\vec k|} }{2} \exp(\ii \vec k \cdot \vec x + \ii c |\vec k| t) \right )   \label{eq:solutionfourierpfunc}
 \end{align}
 \begin{align}
 \vec v(t, \vec x) = \frac{1}{(2\pi)^{d/2}} \int \dd \vec  k &\left( \frac{ \hat p_0(\vec k) + \frac{\vec k \cdot  \hat {\vec v}_0(\vec k)}{|\vec k|} }{2} \frac{\vec k}{|\vec k|} \exp(\ii \vec k \cdot \vec x - \ii c |\vec k| t)  
   \right . \nonumber\\&+ \left.
 \frac{ \frac{\vec k \cdot  \hat {\vec v}_0(\vec k)}{|\vec k|} - \hat p_0(\vec k) }{2} \frac{\vec k}{|\vec k|} \exp(\ii \vec k \cdot \vec x + \ii c |\vec k| t) \right. \nonumber\\&+ \left. \left\{  \hat {\vec v}_0(\vec k) - \frac{\vec k}{|\vec k|} \frac{\vec k \cdot \hat {\vec v}_0(\vec k)}{|\vec k|}   \right \} \exp(\ii \vec k \cdot \vec x)  \right )  \label{eq:solutionfouriervfunc}
\end{align}
{where ${\hat p}_0(\vec k) = \mathbb F_\vec x[p_0](\vec k)$ and ${\hat {\vec v}}_0(\vec k) = \mathbb F_\vec x[\vec v_0](\vec k)$.}

An analogous formula is valid in the sense of distributions:
\begin{theorem}
 The distributional solution to the initial value problem given by \eqref{eq:acousticv}--\eqref{eq:acousticp} {and initial data $(\vec v_0, p_0) \in (S'(\mathbb R^d))^m$ with corresponding Fourier transforms $(\hat {\vec v}_0, \hat p_0) \in (S'(\mathbb R^d))^m$}, $m=d+1$, is %\right. \\ &+ \left .
 \begin{align}
 p(t, \vec x) = \mathbb F^{-1}_{\vec k}\left[\frac{\hat p_0(\vec k) + \frac{\vec k \cdot  \hat {\vec v}_0(\vec k)}{|\vec k|} }{2}  \exp(- \ii c |\vec k| t) 
 +\frac{\hat p_0(\vec k) - \frac{\vec k \cdot  \hat {\vec v}_0(\vec k)}{|\vec k|} }{2} \exp(\ii c |\vec k| t) \right ](\vec x) \label{eq:solutionfourierp}
 \end{align}
 \begin{align}
 \vec v(t, \vec x) = \mathbb F^{-1}_{\vec k}&\left[ \frac{ \hat p_0(\vec k) + \frac{\vec k \cdot  \hat {\vec v}_0(\vec k)}{|\vec k|} }{2} \frac{\vec k}{|\vec k|} \exp( - \ii c |\vec k| t) +
 \frac{ \frac{\vec k \cdot  \hat {\vec v}_0(\vec k)}{|\vec k|} - \hat p_0(\vec k) }{2} \frac{\vec k}{|\vec k|} \exp( \ii c |\vec k| t) \right. \nonumber\\&+ \left. \left\{  \hat {\vec v}_0(\vec k) - \frac{\vec k}{|\vec k|} \frac{\vec k \cdot \hat {\vec v}_0(\vec k)}{|\vec k|}   \right \}   \right ] (\vec x)\label{eq:solutionfourierv}
\end{align}

\end{theorem}
{\emph{Note}: In the smooth case the distributional solution \eqref{eq:solutionfourierp}--\eqref{eq:solutionfourierv} reduces to \eqref{eq:solutionfourierpfunc}--\eqref{eq:solutionfouriervfunc}.}
\begin{proof}
The use of $S'$ makes sure that the Fourier transforms exist according to Theorem \ref{thm:fourierautom}. Denoting the solution $q = (\vec v, p)$ (with $m=d+1$ components and independent variables $t$, $\vec x$) and its Fourier transform $\widehat{q(t)}$ {with respect to $\vec x$} (with independent variable $\vec k$), one has $q(t) = \mathbb F^{-1}_{\vec k} [\widehat{q(t)}]$. $q$ being the distributional solution to the system $\del_t q + \vec J \cdot \nabla q = 0$ of PDEs means
\begin{align*}
 \Big\langle (\del_t + \vec J \cdot \nabla) \mathbb F^{-1}_{\vec k}[\widehat{q(t)}] \Big| \test \Big\rangle &= 0 \qquad \forall \test \in {(S(\mathbb R^{{d}}))^m}
\end{align*}
which by \eqref{example-diffdistr} is
\begin{align}
 \Big \langle \mathbb F^{-1}_{\vec k}\left[\left({\id} \frac{\dd}{\dd t}  + \ii \vec J \cdot \vec k\right)\widehat{q(t)}\right] \Big| \test\Big  \rangle &= 0 \label{eq:intermediatefourieracoustic}
\end{align}
The matrix $ \vec J \cdot \vec k$ appears in the study of the Cauchy problem and bicharacteristics (see e.g. \cite{courant62}, VI, \S3). Hyperbolicity guarantees its real diagonalizability with eigenvalues $\omega_n \in \mathbb R$ ($n=1,\ldots m$). For the acoustic system \eqref{eq:acousticv}--\eqref{eq:acousticp} $\vec J \cdot \vec k$ is symmetric and $\omega_n \in \{0, \pm c |\vec k|\}$. Choosing orthonormal eigenvectors $e_n$ ($n=1, \ldots, m$) which fulfill
\begin{align*}
 (\vec J \cdot \vec k ) e_n = \omega_n e_n
\end{align*}
the vector $\widehat{q(0)}$ can be, for every $\vec k$, decomposed according to the eigenbasis of $\vec J \cdot \vec k$:
\begin{align*}
  \widehat{q(0)}(\vec k) = \sum_{n=1}^{m} e_n ( e_n \cdot {\widehat{q(0)}}(\vec k))
\end{align*}

Equation \eqref{eq:intermediatefourieracoustic} then is obviously solved by 
\begin{align*}
  \widehat{q(t)}(\vec k) = \sum_{n=1}^{m} e_n ( e_n \cdot {\widehat{q(0)}}(\vec k)) \exp(-\ii \omega_n(\vec k) t)
\end{align*}
and transforming back
\begin{align*}
 q(t, \vec x) &= \sum_{n=1}^{m} \mathbb F^{-1}_{\vec k} [ \hat q_{0n}(\vec k) \exp(-\ii \omega_n(\vec k) t) ]
\end{align*}
in the sense of distributions with
\begin{align*}
 \hat q_{0n}(\vec k) = e_n \Big( e_n \cdot \mathbb F_{\vec x}[q_0](\vec k) \Big)
\end{align*}
Using this, and computing the eigenspace projectors explicitly for \eqref{eq:acousticv}--\eqref{eq:acousticp} yields \eqref{eq:solutionfourierp}--\eqref{eq:solutionfourierv}.

{Consider the $m=d+1$ equations \eqref{eq:solutionfourierp}--\eqref{eq:solutionfourierv} as a mapping $ {\mathbb R^+_0} \to (S'(\mathbb R^d))^m$, i.e. $t \mapsto q(t)$. Then, for $\psi \in (S(\mathbb R^d))^m$, $t \mapsto \langle q(t) | \psi \rangle =  \langle \widehat{q(t)} | \mathbb F_{\vec k}^{-1}[ \psi] \rangle$ is smooth, in particular $q \in C^1(\mathbb R^+_0,(S'(\mathbb R^d))^m)$. For $t=0$, \eqref{eq:solutionfourierp}--\eqref{eq:solutionfourierv} obviously yield the initial data.}
\end{proof}

Given the Fourier transform of any initial data therefore the solution can easily be constructed. The solution formulae are most conveniently expressed using spherical averages as given {in Definition \ref{def:sphericalaverage}. With Theorem \ref{thm:fourierradial} it is possible to derive explicit solution formulae} for \eqref{eq:acousticv}--\eqref{eq:acousticp}.

\begin{theorem}[Solution formulae] \label{thm:solutionappendix}
 Consider the distributions ({$d=3$, $m=d+1$})
 \begin{align*} q(t, \vec x) = (\vec v(t, \vec x), p(t, \vec x)) \in C^1(\mathbb R^+_0, (S'(\mathbb R^{{d}}))^{m}) \end{align*}
 with
 \begin{align}
 p(t, \vec x) &= \regdist{p_0(\vec x)}  - \frac{1}{4 \pi} \frac{1}{ct} (\div \regdist{{\vec v}_0} * \delta_{|\vec x|=ct}) - \frac{1}{4 \pi} (\div \grad  \regdist{p_0} * \regdist{\frac{\Theta_{|\vec x|\leq ct}}{|\vec x|}}) \label{eq:phelmholtz}\\
 \vec v(t, \vec x) &= \regdist{{\vec v}_0(\vec x)} + \frac{1}{4 \pi} (\grad \div \regdist{{\vec v}_0} * \regdist{\frac{\Theta_{|\vec x|\leq ct}}{|\vec x|}}) - \frac{1}{4\pi} \frac{1}{ct}(\grad \regdist{p_0} * \delta_{|\vec x|=ct}) \label{eq:vhelmholtz}
 \end{align}
 They are distributional solutions to
  \begin{align*}
    \del_t q + \vec J \cdot \nabla q &= 0
  \end{align*}
  with $\vec J$ given by \eqref{eq:acousticjacobians} and ${L^1_\text{loc} \cap L^\infty}$ initial data $\vec v_0(\vec x), p_0(\vec x)$ such that \begin{align*}(\vec v(0, \vec x), p(0, \vec x)) = (\regdist{\vec v_0(\vec x)}, \regdist{ p_0(\vec x)}) \in (S'(\mathbb R^d))^{m} \end{align*}
  
\end{theorem}
\begin{proof}
 Recall the differentiation rule in presence of the Fourier transform as formulated in \eqref{eq:fourierderiv}. Denote by $\del_i$ differentiation with respect to the $i$-the direction and $k_i$ the corresponding component of $\vec k$.

 Inserting the definition of ${\hat p}_0(\vec k)$ and $\hat {\vec v}_0(\vec k)$ into \eqref{eq:solutionfourierp}--\eqref{eq:solutionfourierv} yields
 \begin{align*}
 \mathbb F_{\vec x} [p(t, \vec x)](\vec k) &= \mathbb F_{\vec x}[\regdist{p_0}](\vec k)\cdot \cos(c |\vec k| t) - \mathbb F_{\vec x}[\div \regdist{{\vec v}_0}](\vec k)\cdot \frac{\sin( c |\vec k| t) }{|\vec k|}\\
 \mathbb F_{\vec x}[\vec v(t, \vec x)](\vec k) &= \mathbb F_{\vec x}[\regdist{{\vec v}_0}](\vec k)  - \mathbb F_{\vec x}[ \grad \div \regdist{{\vec v}_0}](\vec k) \frac{ \cos(  c |\vec k| t)-1 }{|\vec k|^2}\nonumber\\ &- \mathbb F_{\vec x}[\grad \regdist{p_0}](\vec k) \, \frac{\sin(c |\vec k| t)}{|\vec k|}
\end{align*}
Now using Theorem \ref{thm:fourierradial} one rewrites
 \begin{align*}
 \mathbb F_{\vec x} [p(t, \vec x)](\vec k) &=\mathbb F_{\vec x}[\regdist{p_0}](\vec k)  - \mathbb F_{\vec x}[\div \regdist{{\vec v}_0}](\vec k)\cdot \mathbb F_{\vec x}[\delta_{|\vec x|=ct}](\vec k) \frac{\sqrt{2\pi}}{2 ct}\nonumber\\
 &- \mathbb F_{\vec x}[\div \grad  \regdist{p_0}](\vec k)\cdot \mathbb F_{\vec x}\left[\regdist{\frac{\Theta_{|\vec x|\leq ct}}{|\vec x|}}\right](\vec k) \frac{\sqrt{2\pi}}{2}\\
 \mathbb F_{\vec x}[\vec v(t, \vec x)](\vec k) &= \mathbb F_{\vec x}[\regdist{{\vec v}_0}](\vec k)  + \mathbb F_{\vec x}[ \grad \div \regdist{{\vec v}_0}](\vec k) \cdot \mathbb F_{\vec x}\left[\regdist{\frac{\Theta_{|\vec x|\leq ct}}{|\vec x|}}\right](\vec k) \frac{\sqrt{2 \pi}}{2}\nonumber\\ &- \mathbb F_{\vec x}[\grad \regdist{p_0}](\vec k) \cdot \mathbb F_{\vec x}[\delta_{|\vec x|=ct}](\vec k) \frac{\sqrt{2\pi}}{2 ct}
\end{align*}
When rewriting $\cos(c |\vec k| t) -1$ as a Fourier transform, $1 = \frac{\vec k \cdot \vec k}{|\vec k|^2}$ has been inserted. As both $\regdist{\frac{\Theta_{|\vec x|\leq ct}}{|\vec x|}}$ and $\delta_{|\vec x|=ct}$ have compact support, the convolutions that involve one of them are well defined (see Definition \ref{def:convolution}). Thus the products above can be converted into Fourier transforms of convolutions, which proves the assertion.
\end{proof}

\begin{corollary} \label{cor:helmholtzappendix}
 {For $C^2 \cap L^\infty$ initial data $p_0$, $\vec v_0$}, Equations \eqref{eq:phelmholtz}--\eqref{eq:vhelmholtz} become
 \begin{align}
   p(t, \vec x) &= p_0(\vec x) + \int_0^{ct} \dd r \, r \frac{1}{4\pi} \int_{S^2_1} \dd {\vec y} \,\,(\div \grad p_0)(\vec x + r \vec y) - ct \frac{1}{4\pi}\int_{S^2_1} \dd {\vec y}\,\, \div \vec v_0 (\vec x + ct \vec y)\label{eq:phelmholtzfuncappendix}\\
 \vec v(t, \vec x) &= \vec v_0(\vec x) + \int_0^{ct} \dd r \, r \frac{1}{4\pi}\int_{S^2_1} \dd {\vec y} \,\, (\grad \div \vec v_0)(\vec x + r \vec y) - ct \frac{1}{4\pi}\int_{S^2_1} \dd {\vec y} \,\, (\grad p_0)(\vec x + ct \vec y)\label{eq:vhelmholtzfuncappendix}
  \end{align}
\end{corollary}
\begin{proof}
The formulae \eqref{eq:phelmholtz}--\eqref{eq:vhelmholtz} are transformed into \eqref{eq:phelmholtzfuncappendix}--\eqref{eq:vhelmholtzfuncappendix} by noting that if $f$ is integrable, then for all $\test \in D(\mathbb R^d)$
\begin{align*}
 \frac{1}{4 \pi} \left(f * \frac{\Theta_{|\vec x|\leq ct}}{|\vec x|}\right) &= \frac{1}{4 \pi}  \int_{|\vec y| \leq ct} \dd \vec y  \frac{1}{|\vec y|}  f(\vec x - \vec y) = \int_0^{ct}  \dd r \,r \frac{1}{4\pi}\int_{S^2_1} \dd {\vec y}  f(\vec x + r \vec y)\\
 \frac{1}{4 \pi} \langle \regdist{f} * \delta_{|\vec x|=ct}| \test \rangle &= \Big \langle \frac{1}{4 \pi}  \int_{S^2_{ct}} \dd \vec y    f(\vec x - \vec y) \Big | \test \Big \rangle = \Big \langle (ct)^2 \frac{1}{4\pi}\int_{S^2_1} \dd {\vec y}  f(\vec x + ct \vec y) \Big | \test \Big \rangle
\end{align*}

\end{proof}

{The components of any vector $\vec y \in \mathbb R^d$ are denoted by $y_i$, $i = 1, \ldots, d$. For example the components of the unit normal vector $\vec n := \frac{\vec x}{|\vec x|}$ are denoted by $n_i$, $i = 1, \ldots, d$. The Kronecker symbol is
\begin{align*}
 \delta_{ij} := \begin{cases} 1 & i = j \\ 0 & \text{else} \end{cases}
\end{align*}
}

In order to simplify notation the following distributions will be used:

\begin{theorem} \label{def:sigmaij}
 \begin{enumerate}[i)] 
  \item Given a test function $\test \in D(\mathbb R^3)$ the following integral exists
  \begin{align}
   \int_0^{ct} \dd r \frac{1}{r^3} \int_{S^2_r} \dd \vec y \left( 3\frac{y_i y_j}{|\vec y|^2} - \delta_{ij} \right ) \test(\vec y)   \label{eq:defdistrSigma}
  \end{align}
  This defines a distribution $\Sigma_{ij}(ct)$ whose action $\langle \Sigma_{ij}(ct) | \test \rangle$ onto a test function $\test$ is given by \eqref{eq:defdistrSigma}.
  \item Given a test function $\test \in D(\mathbb R^3)$ the following integral exists
 \begin{align}
 \int_0^{ct} \dd r \frac{1}{r} \del_r \left[ \frac{1}{r} \del_r \left( r \int_{S^2_r} \dd \vec y \frac{y_i y_j}{|\vec y|^2} \test(\vec y) \right )  -  \frac{1}{r} \int_{S^2_r} \dd \vec y  \delta_{ij} \test(\vec y)   \right ]   \label{eq:defdistrsigma}
 \end{align}
 This defines a distribution $\sigma_{ij}(ct)$ whose action $\langle \sigma_{ij}(ct) | \test \rangle $ onto a test function $\test$ is given by \eqref{eq:defdistrsigma}.
 \end{enumerate}
\end{theorem}
\begin{proof}
One needs to prove the existence of the integrals, because including the origin into the integration domain might potentially be problematic. Therefore, for {$a \in \mathbb R^+$}, divide the integration over $[0, ct]$ into two integrals over $[0, a]$ and $[a, ct]$ and consider $a \to 0$.

The reason why the integrals exist is subtle. First of all, without the test function one finds upon explicit computation
 \begin{align}
  \int_{S^2_1} \dd \vec y \,  y_i y_j = \frac13 \int_{S^2_1} \dd \vec y  \delta_{ij} = \frac{4\pi}{3} \delta_{ij} \label{eq:concellationtensornomral}
 \end{align}
 Precisely this combination $3 y_i y_j - \delta_{ij}$ appears in \eqref{eq:defdistrSigma}.
 
\begin{enumerate}[i)]
 \item \label{it:firstsigma} 
 Recall that $\test \in C^\infty$, and therefore by the mean value theorem there exists $\xi(r, \vec y)$ such that
 \begin{align*}
  \test(r \vec y) = \test(0) + r \vec y \cdot\nabla \test(\xi \vec y)
 \end{align*}
 Then for $a > 0$ 
 \begin{align*}
  \int_0^{a}\dd r \, \frac{1}{r^3} \int_{S^2_r} \dd \vec y & \, \left(3 \frac{y_i y_j}{|\vec y|^2} - \delta_{ij} \right ) \test(\vec y) 
  \overset{S_r^2 \leftrightarrow S_1^2}{=} \int_0^{a}\dd r \, \frac{1}{r} \int_{S^2_1} \dd \vec y \, (3 y_i y_j - \delta_{ij}) \test(r\vec y)\\
  &= \int_0^{a}\dd r \, \frac{1}{r} \int_{S^2_1} \dd \vec y \, (3 y_i y_j - \delta_{ij}) \Big(\test(0) + r \vec y \cdot\nabla \test(\xi \vec y) \Big)
 \end{align*}
 \begin{align*}
  \left|  \int_0^{a}\dd r \, \frac{1}{r^3} \int_{S^2_r} \dd \vec y \, \left(3 \frac{y_i y_j}{|\vec y|^2} - \delta_{ij} \right ) \test(\vec y) \right | &\overset{\eqref{eq:concellationtensornomral}}{=}
  \left| \int_0^{a}\dd r \,  \int_{S^2_1} \dd \vec y \, (3 y_i y_j - \delta_{ij})  \vec y \cdot \nabla \test(\xi \vec y) \right |
  \\&\leq C a \| \nabla \test \|_\infty
 \end{align*}
 \item 
 By expanding
 \begin{align*}
  \langle \sigma_{ij}(ct) | \test \rangle &:= \int_0^{ct} \dd r \frac{1}{r} \del_r \left[ \frac{1}{r} \del_r \left( r^3 \int_{S^2_1} \dd \vec y y_i y_j \test(r\vec y) \right )  -  r \int_{S^2_1} \dd \vec y  \delta_{ij} \test(r\vec y)   \right ] \\
  &= \int_0^{ct} \dd r   \left[ 
  \frac{1}{r} \int_{S^2_1} \dd \vec y \, (3 y_i y_j - \delta_{ij}) \test(r\vec y) 
  + 5 \del_r \int_{S^2_1} \dd \vec y \, (5 y_i y_j - \delta_{ij}) \test(r\vec y)  \right. \nonumber\\& \phantom{mmmmm}\left. 
  + r \del_r^2 \int_{S^2_1} \dd \vec y \,  y_i y_j \test(r\vec y)  
  \right ]
 \end{align*}
 This reduces to the case discussed in \ref{it:firstsigma}).
 \end{enumerate}
\end{proof}

In the following the convention of summing over repeated indices is adopted. {Recall the definition of the radial derivative $\del_r := \vec n \cdot \nabla$ with $\vec n = \frac{\vec x}{|\vec x|}$.}

\begin{lemma}  \label{lem:differentiationandgauss}
The following laws of differentiation allow to transfer derivatives when working with spherical means:
 \begin{enumerate}[i)]
  \item \label{it:lemdifferentiationandgaussradial} 
  \begin{align}
  \del_r \left( \regdist{f} *\, \delta_{|\vec x|=r} \frac{1}{r^2} \right ) &= \del_i \regdist{f}(x) * \delta_{|\vec x|=r} \frac{1}{r^2} n_i   \label{eq:diffrlemma}
  \end{align}
 \item \label{it:lemdifferentiationandgaussgauss} 
 \begin{align}
  \del_i (\regdist{f} * \,\delta_{|\vec x|=r} n_i) &= -\nabla_{\vec x} \cdot \nabla_{\vec x} \regdist{f}(x) * \regdist{\Theta_{|\vec x| \leq r}} \label{eq:gausslemma}
 \end{align}
 \end{enumerate}
\end{lemma}
\begin{proof}
 \begin{enumerate}[i)]
  \item Consider a test function $\test \in D(\mathbb R^3)$:
  \begin{align*}
  \langle \del_r \regdist{f} *\, \delta_{|\vec x|=r} \frac{1}{r^2} | \test \rangle &= \Big \langle  \regdist{f}(\vec x) | \del_r \langle \delta_r(\vec y) \frac{1}{|\vec y|^2} | \test(\vec x + \vec y) \rangle \Big \rangle = - \Big \langle \regdist{f}(\vec x) |  \int_{S^2_1} \dd \vec y \del_r \test(\vec x + r\vec y) \Big \rangle\\
  &= - \langle \regdist{f}(\vec x) | \int_{S^2_1} \dd \vec y \, \vec y \cdot \nabla_{\vec x} \test(\vec x + r\vec y)  \rangle\\
  &= \left \langle \nabla_{\vec x} \regdist{f}(\vec x) \Big | \int_{S^2_r} \dd \vec y \frac{1}{|\vec y|^2} \frac{\vec y}{|\vec y|} \cdot  \test(\vec x + \vec y)  \right \rangle\\
  &= \left \langle \frac{\del}{\del x_i} \regdist{f}(x) * \delta_{|\vec x|=r} \frac{1}{r^2} n_i \Big| \test \right \rangle \end{align*}
 \item Recall Theorem \ref{thm:deriavtiveradialstep} which states that $-\delta_{|\vec x|=r} \vec n = \nabla \regdist{\Theta_{|\vec x| \leq r}}$. Thus 
 \begin{align*}
    \left \langle \del_i (\regdist{f} * \delta_{|\vec x|=r} n_i)\Big | \test \right\rangle  &=-\left \langle \regdist{f} \Big | \left \langle \delta_{|\vec x|=r} n_i | \del_i \test \right \rangle \right \rangle\\
    &= \left \langle \regdist{f} \Big | \left \langle \del_i \regdist{\Theta_{|\vec x| \leq r}} | \del_i \test \right \rangle \right \rangle 
    = -\left \langle \regdist{f} \Big | \left \langle \del_i \del_i \regdist{\Theta_{|\vec x| \leq r}} | \test \right \rangle \right \rangle \\
    &= -\left \langle \nabla_{\vec x} \cdot \nabla_{\vec x} \regdist{f}(x) * \regdist{\Theta_{|\vec x| \leq r}} \Big | \test \right \rangle 
 \end{align*} 
 \end{enumerate}
\end{proof}
{\textit{Note}}: Equation \eqref{eq:diffrlemma} is the distributional analogue to \begin{align*}\del_r \int_{S^2_1} \dd \vec y \, f(\vec x + r \vec y) = \int_{S^2_1} \dd \vec y \, y_i \del_i f(\vec x + r \vec y)\end{align*} and Equation \eqref{eq:gausslemma} is the distributional analogue to \begin{align*} r^2 \int_{S^2_1} \dd {\vec y} \, \vec y \cdot \nabla p_0(\vec x + r\vec y)  = \int_0^r \dd r' \, {r'}^2 \int_{S^2_1} \dd {\vec y} \, \nabla \cdot \nabla p_0(\vec x + r'\vec y) \end{align*}

\begin{theorem}[Solution formulae with radial derivatives only]  \label{thm:solutionradialappendix}
Consider the setup of Theorem \ref{thm:solutionappendix} and write the components of the initial data $\vec v_0$ as $v_{0i}$, $i  = 1, 2, 3$. The solution \eqref{eq:phelmholtz}--\eqref{eq:vhelmholtz} can be rewritten as
\begin{align}
 p(t, \vec x) &= \del_r \left(\frac1{4\pi}\frac{\delta_{|\vec x|=r}}{r}*\regdist{p_0}\right) - \frac{1}{r} \del_r \left(\frac1{4\pi} \delta_{|\vec x|=r} n_i * \regdist{v_{0i}}\right ) \label{eq:preformulappendix}\\
 v_j(t, \vec x) &= \frac23 \regdist{ v_{0j}}(\vec x) - \frac{1}{r} \del_r \left( \frac1{4\pi}\delta_{|\vec x|=r} n_j * \regdist{p_0} \right) + \del_r \left(\frac1{4\pi}\frac{\delta_{|\vec x|=r}}{r} n_i n_j * \regdist{v_{0j}} \right ) \nonumber \\&- \left( \frac1{4\pi}\frac{\delta_{|\vec x|=r}}{r^2} (\delta_{ij} - 3 n_i n_j) * \regdist{v_{0i}}\right) + \frac1{4\pi}\Sigma_{ij}(ct) * \regdist{v_{0i}}   \label{eq:vsolution23appendix}
\end{align}
where all derivatives are to be evaluated at $r=ct$.

Equation \eqref{eq:vsolution23appendix} is equivalent to
\begin{align}
 v_j(t, \vec x) = \regdist{v_{0j}}(\vec x) &- \frac{1}{r}\del_r \left( \frac{1}{4\pi} \delta_{|\vec x|=r} n_j * \regdist{p_0}  \right) + \frac{1}{4\pi} \sigma_{ij}(ct) * \regdist{v_{0i}} \label{eq:ureformulappendix}
\end{align}

\end{theorem}
{\sl Note}: The convolutions above let \emph{Green's kernel} for \eqref{eq:acousticv}--\eqref{eq:acousticp} appear clearly.
\begin{proof}

In order to transfer the $r$-derivatives in \eqref{eq:preformulappendix}--\eqref{eq:ureformulappendix} into the derivative operators in \eqref{eq:phelmholtz}--\eqref{eq:vhelmholtz} one uses the Gauss theorem for the sphere of radius $r$. For example, differentiating
\begin{align*}
 \del_r \left( \regdist{p_0} * \frac{\delta_{|\vec x|=r}}{r} \right )  
\end{align*}
with respect to $r$ yields
\begin{align*}
 \del_r^2 \left( \regdist{p_0} * \frac{\delta_{|\vec x|=r}}{r} \right ) &= \frac{1}{r} \del_r \left( r^2 \del_r \left( \regdist{p_0} * \frac{\delta_{|\vec x|=r}}{r^2} \right ) \right ) 
\intertext{by elementary manipulations. According to Lemma \ref{lem:differentiationandgauss}\,\ref{it:lemdifferentiationandgaussradial}), differentiation with respect to $r$ can be replaced by $\vec n \cdot \nabla$ inside the spherical mean:}
 &= \frac{1}{r} \del_r \left( r^2   \left( \frac{\del}{\del x_i} \regdist{p_0} * \frac{\delta_{|\vec x|=r} n_i}{r^2} \right ) \right ) 
\intertext{and by Gauss theorem (Lemma \ref{lem:differentiationandgauss}\,\ref{it:lemdifferentiationandgaussgauss})) as well as Theorem \ref{thm:deriavtiveradialstep}}
 &=  - \frac{1}{r} \del_r \left ( \nabla_{\vec x} \cdot \nabla_{\vec x} \regdist{p_0} * \regdist{\Theta_{|\vec x| \leq r}} \right ) 
 =   \frac{1}{r} \left ( \nabla_{\vec x} \cdot \nabla_{\vec x} \regdist{p_0} * \delta_{|\vec x| = r} \right ) 
\end{align*}
Integrating over $r$, and evaluating at $r=ct$ yields the sought identity
\begin{align*}
\left. \del_r \left( \regdist{p_0} * \frac{\delta_{|\vec x|=r}}{r} \right ) \right |_{r=ct} &= \regdist{p_0} +  \nabla \cdot \nabla \regdist{p_0} * \regdist{\frac{\Theta_{|\vec x| \leq ct}}{|\vec x|}} 
\end{align*}

In a similar way the equivalence of the other terms can be shown and is omitted here.

\end{proof}

\end{document}